\documentclass[a4paper]{article}

\usepackage{import}
\usepackage{setspace}
\usepackage{framed}
\usepackage{color}
\usepackage{a4,amsmath,amssymb,amscd,latexsym} 
\usepackage{amsthm}
\usepackage{bbm}
\usepackage{epsfig,color} 
\usepackage{amsfonts}
\usepackage{mathrsfs}
\usepackage{overpic}
\usepackage{subcaption}
\usepackage{paralist}
\usepackage{graphicx}
\usepackage{trfsigns}
\usepackage{url}
\usepackage{stmaryrd}
\usepackage{booktabs}
\usepackage{algorithm, algorithmic}
\usepackage[algo2e, boxed, lined]{algorithm2e} 
\usepackage{hyperref}
\usepackage{cleveref}
\usepackage{scalerel}
\usepackage[export]{adjustbox}
\usepackage{subcaption}
\usepackage{eufrak}

\usepackage{tikz, pgfplots}
\usetikzlibrary{calc,trees,positioning,arrows,chains,shapes.geometric,%
	decorations.pathreplacing,decorations.pathmorphing,shapes,%
	matrix,shapes.symbols,backgrounds,shadows}

\usepackage{tikz-cd}

\newtheorem{definition}{Definition}
\newtheorem{corollary}{Corollary}
\newtheorem{theorem}{Theorem}

\newtheorem{remark}{Remark}
\newtheorem{proposition}{Proposition}

\newcommand\fat[1]{\ThisStyle{\ooalign{%
			\kern.46pt$\SavedStyle#1$\cr\kern.33pt$\SavedStyle#1$\cr%
			\kern.2pt$\SavedStyle#1$\cr$\SavedStyle#1$}}}

\newenvironment{@abssec}[1]{%
	\vspace{.06in}\footnotesize
	\parindent=0in
	\ignorespaces 
}

\newenvironment{keywords}{\begin{@abssec}{\keywordsname}}{\end{@abssec}}
\newenvironment{AMS}{\begin{@abssec}{AMS subject classification:}}{\end{@abssec}}

\title{\bf{Pre-Shape Calculus: Foundations and Application to Mesh Quality Optimization}}

\author{Daniel Luft\thanks{Trier University, Department of Mathematics, 54286 Trier, Germany (\tt luft@uni-trier.de)}
 \and Volker Schulz\thanks{Trier University, Department of Mathematics, 54286 Trier, Germany (\tt volker.schulz@uni-trier.de)}
		} 

\newcommand{\R}{{\mathbb{R}}} 
\newcommand{\Manifold}{M}
\newcommand{\HoldAll}{\mathbb{D}}
\newcommand{\Diffeomorphism}{\rho}

\newcommand{\Shape}{\Gamma}
\newcommand{\ShapeEmbedding}{\varphi}

\newcommand{\TargetPrShp}{\mathfrak{J}}
\newcommand{\TargetShp}{\mathcal{J}}
\newcommand{\PrShpDeriv}{\mathfrak{D}}
\newcommand{\ShpDeriv}{\mathcal{D}}

\newcommand{\TangentSpaceShape}{\mathcal{T}}
\newcommand{\TangentVector}{\tau}
\newcommand{\NormalSpaceShape}{\mathcal{N}}

\newcommand{\diff}{\mathrm{d}}

\newcommand{\ShapeAdmissibleSet}{\mathcal{A}}
\newcommand{\ProjectionCanonical}{\pi}

\newcommand{\RieszEnergyExtForce}{q}

\begin{document}

%
%
%
%
%

\maketitle
\begin{abstract}
\noindent
Deformations of the computational mesh arising from optimization routines usually lead to decrease of mesh quality or even destruction of the mesh. 
We propose a theoretical framework using pre-shapes to generalize classical shape optimization and calculus. 
We define pre-shape derivatives and derive according structure and calculus theorems.
In particular, tangential directions are featured in pre-shape derivatives, in contrast to classical shape derivatives featuring only normal directions. 
Techniques from classical shape optimization and -calculus are shown to carry over to this framework. 
An optimization problem class for mesh quality is introduced, which is solvable by use of pre-shape derivatives.
This class allows for simultaneous optimization of classical shape objectives and mesh quality without deteriorating the classical shape optimization solution. 
The new techniques are implemented and numerically tested for 2D and 3D. 
\end{abstract}





\section{Introduction}

Solutions of PDE constrained optimization problems, in particular problems where the desired control variable is a geometric shape, are only meaningful, if the state variables of the constraint can be calculated with sufficient reliability.
A key component of reliable solutions is given by quality of the computational mesh. 
It is well-known that poor quality of elements affect the stability, convergence, and accuracy of finite element and other solvers.

We propose a unified framework using so-called pre-shapes.
In this setting, both shape optimization and mesh quality optimization problems can be formulated at the same time. 
We give a class of problems called pre-shape parameterization tracking problems, which can act as regularizations for shape optimization problems.
These problems can be solved for volume and surface meshes with arbitrary dimension, and yield numerical algorithms similar to so called mesh deformation methods, which reallocate nodes of numerical meshes according to targeted element volumes to improve mesh quality.
At the same time, the proposed framework is suitable to derive a calculus mimicking classical shape calculus. 
This enables formulation of efficient routines solving shape optimization problems, which at the same optimize quality of the surface mesh representing the shape, without noticeable additional computational cost or interference with the original shape optimization problem.
With this, desired surface and surrounding volume mesh quality are ensured during shape optimization.

In this paper we will establish the theoretical foundations of pre-shape optimization and calculus, and show its connection to classical shape calculus.
We implement these methods in form of a pre-shape gradient descent to achieve a targeted quality of volume and surface meshes without a shape optimization target.
In \cite{luft2020pre2}, we build on achievements of this paper, and introduce theoretical and numerical results to solve shape optimization problems while simultaneously improving volume and shape mesh quality according to targeted node distributions.
The techniques elaborated in \cite{luft2020pre2} leave optimal shapes invariant and offer minimal additional numerical costs.
They also permit use of different metrics to represent gradients. We compare pre-shape mesh regularizations for various metrics in  \cite{luft2020pre2} for a hard to solve shape optimization problem.

\paragraph{Literature Review}
We give a brief overview of techniques related to the ones treated in this article.
Our methods are not related to mesh untangling and -relaxation, edge swapping or remeshing strategies such as  \cite{freitag1997combining, frey1999surface, algorri1996mesh, johnston1991automatic}. 
Of course, as there is a vast amount of literature concerning mesh generation and improvement strategies thereof, we can only give a selective overview.
Two very prominent classes of techniques for mesh quality improvement are the so called mesh deformation method and methods based on Laplacian smoothing.
Mesh deformation methods go back to a theoretical result initially proposed by Moser in \cite{moser1965volume}, extended by Banyaga \cite{banyaga1974formes} and by Dacorogna and Moser \cite{dacorogna1990partial}. 
This gave rise to mesh deformation methods pioneered in  \cite{liao1992new} by Liao and Anderson, which redistribute mesh vertices such that uniform cell volumes are achieved. 
Mesh deformation methods are powerful, because they prevent mesh tangling while offering precise control over the element volumes. 
The original method was further developed in various directions by Liao et. al \cite{bochev1996analysis, liu1998adaptive, cai2004adaptive, zhou2017novel}.
These advances allow to target non-uniform cell volume distributions, make deformation methods applicable to time-dependent problems, and shows its use in higher order mesh generation methods.
Also, the combination of multigrid- and mesh deformation methods was analyzed and implemented by Turek et. al. in \cite{wan2006numerical, grajewski2009mathematical, grajewski2010numerical}. 
A different family of mesh quality improvement techniques are those based on Laplacian smoothing \cite{field1988laplacian, freitag1997combining, shontz2003mesh, zhang2009surface}.
They do not necessarily track for cell volume distributions, but instead improve quality by averaging or smoothing vertex coordinates more or less specifically.
Several strategies for increasing mesh quality not based on mesh deformation methods mentioned or Laplacian smoothing exist.
In the context of shape optimization and -morphing, these include correcting for errors in Hadamard's theorem due to discretization \cite{etling2018first}, adding non-linear advection terms in shape gradient representations \cite{onyshkevych2020mesh}, approximating shape morphing by volume-preserving mean-curvature flows \cite{laurain2020optimal}, and use of techniques related to centroidal Voronoi reparameterization in combination with eikonal equation based non-linear advection terms for representation of shape gradients  \cite{schmidt2014two}.

\section{General Theory for Pre-Shape Calculus}\label{Section_GeneralTheory}
\paragraph[Pre-Shape Spaces]{Pre-Shape Spaces}
In order to provide theoretical grounds for the numerical procedures we are about to elaborate in the subsequent paper, we need to specify a framework for the objects, 'shapes', for which we seek to optimize. 
Several possible theories and techniques exist in order to precisely formulate the notion of shapes. 
For example, shapes can be viewed as sets together with corresponding characteristic functions in an ambient space, leading to an approach which emphasizes geometric measure theory as in \cite{Delfour-Zolesio-2001} by Delfour and Zol\'esio. 
We choose a different setting, namely a shape space approach using infinite dimensional differential geometry, since it naturally permits to view the shape and its parameterization at the same time. 
This is key to extend numerical routines in a way that optimizes parameterizations, i.e. meshes, in a desired way  without interfering with the shape optimization taking place.
For an excellent overview of shape spaces we refer to \cite{bauer2014overview}, from which we borrow several definitions for the following introduction to shape spaces considered in this article.


For the rest of this article, let $\Manifold$ be an $n$-dimensional, orientable, path-connected and compact $C^{k,\alpha}$- or $C^\infty$-manifold.
Further, we will use $\R^{n+1}$ as the ambient space for building our theory.
In particular only shapes of codimension $1$ are considered. 

Denote by $\operatorname{Diff}(\Manifold)$ the regular Lie-group of $C^\infty$-diffeomorphisms of $\Manifold$ onto itself. 
Then the space $B_e$ of unparameterized $C^\infty$-shapes in $\R^{n+1}$ is defined by (cf. \cite{MichorMumford2})
	\begin{equation}\label{Defi_Be}
		B_e(\Manifold, \R^{n+1}) := \operatorname{Emb}(\Manifold, \R^{n+1})/\operatorname{Diff}(\Manifold),
	\end{equation}
where $\operatorname{Emb}(\Manifold, \R^{n+1})$ is the space of all $C^\infty$-embeddings of $\Manifold$ into $\R^{n+1}$ and $\operatorname{Diff}(\Manifold)$ is acting on the right.
For avoidance of confusion, we remember that in this context a $C^\infty$-embedding is an injective, smooth map $\ShapeEmbedding: \Manifold \rightarrow \R^{n+1}$, which has injective first derivative everywhere, i.e. $\ShapeEmbedding$ is an injective immersion.
The resulting space $B_e(\Manifold, \R^{n+1})$, also called non-linear Grassmannian (cf. \cite[44.21ff]{kriegl1997convenient}) or differentiable Chow-variety.
It forms a smooth Hausdorff manifold, whose elements can be regarded as unparameterized hypersurfaces of $\R^{n+1}$ (cf. \cite[Corollary 3.3]{michor2007overview}). 
In the following we will abbreviate $B_e(\Manifold, \R^{n+1})$ by $B_e^n$, still having the implicit relation to the manifold $\Manifold$ and its dimension in mind.
This space can be equipped with various metrics, which means shape optimization can be regarded as optimization on an infinite dimensional Riemannian manifold (cf. \cite{Schulz}).
Notice that $B_e^n$ is not a manifold if $C^\infty$-regularity is replaced by H\"older- or Sobolev-regularity. 
In this more general setting, resulting spaces have a diffeological structure (cf. \cite{welker2017suitable}).
To acquire intuition, a graphical visualization of $B^n_e$ is given in \cref{Fig_BeSpace}.

\begin{figure}[h!]\centering
\vspace{5.5cm}
\def\svgwidth{100.pt}
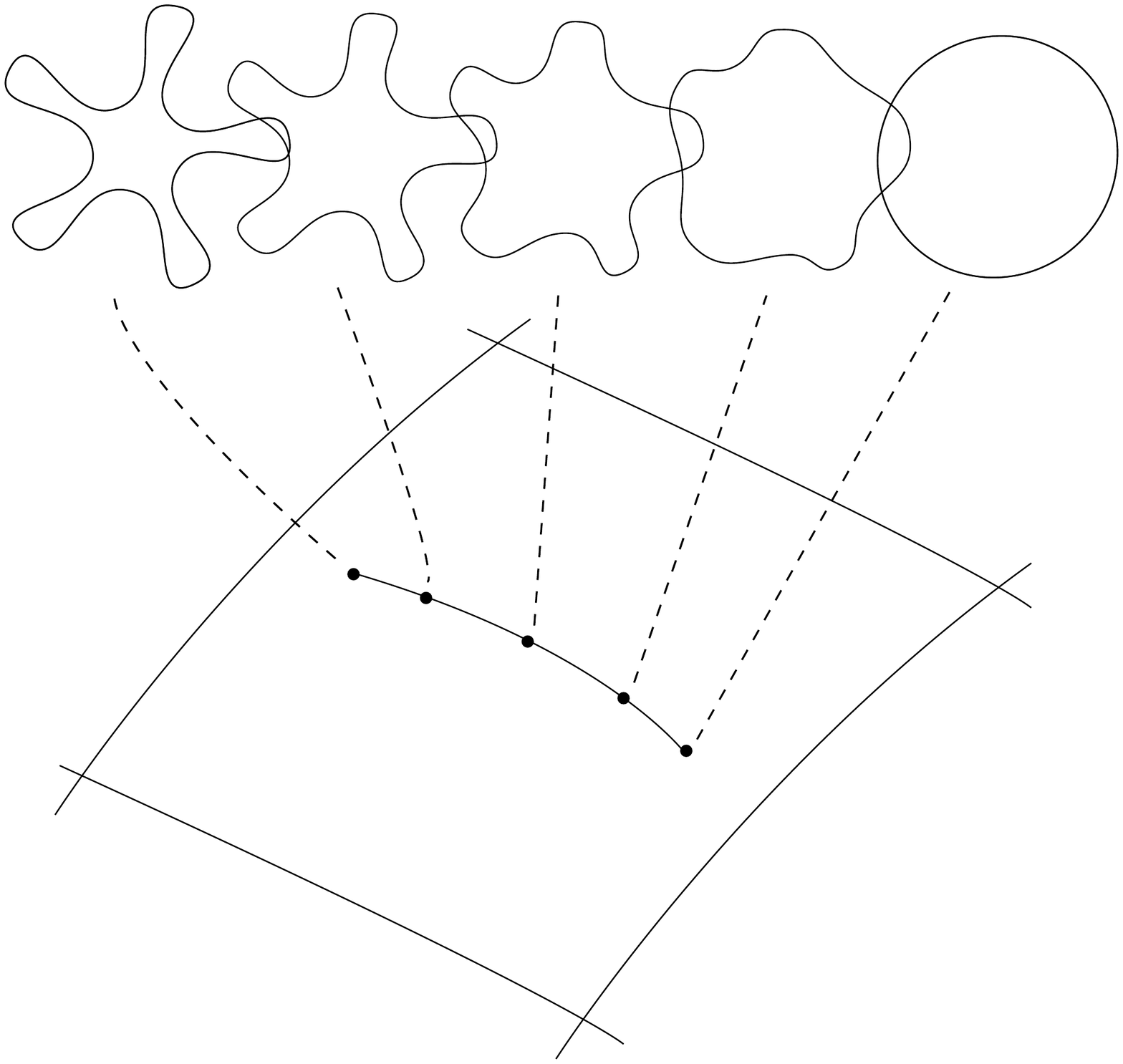
\caption{\label{Fig_BeSpace}Illustration of a path in the shape space $B_e^n$ for $\Manifold=S^1$.}
\end{figure}

For our purposes it is not enough to view shape optimization as optimization in $B_e^n$. 
Instead, we will exploit additional structures on the space $\operatorname{Emb}(\Manifold,\R^{n+1})$ induced by the action of $\operatorname{Diff}(\Manifold)$ and base our framework as optimization in $\operatorname{Emb}(\Manifold,\R^{n+1})$. 
Elements  $\ShapeEmbedding\in\operatorname{Emb}(\Manifold,\R^{n+1})$ can be interpreted as parameterized shapes in the ambient space $\R^{n+1}$, whereas elements of $\operatorname{Diff}(\Manifold)$ acting on the right can be seen as reparameterizations.
The authors of \cite[Ch. 1.1]{bauer2014overview} call $\operatorname{Emb}(\Manifold,\R^{n+1})$ \emph{pre-shape space}, an expression we will borrow for the techniques we will build in this paper.
Notice that the term pre-shape space is used differently depending on the literature, e.g. in \cite{kendall2009shape}, where the authors use this term for the space of labeled landmarks which are equivalent under translation and scaling.

The additional structure of parameterized shapes $\ShapeEmbedding\in \operatorname{Emb}(\Manifold,\R^{n+1})$ compared to unparameterized shapes $\Shape \in B_e^n$ enables to not just view the shape itself, but also to distinguish various types of discretizations in the ambient space and corresponding numerical meshes. 
Even further, this concept enables to control the parameterization of the hold-all domain itself, allowing for control of the way volume meshes are discretized.
The structure for this is given by the fact (cf. \cite[Thrm. 44.1]{kriegl1997convenient}, \cite{binz1981manifold}) that the quotient map 
	\begin{equation}\label{Defi_FiberBundleEmb}
		\ProjectionCanonical: \operatorname{Emb}(\Manifold, \R^{n+1}) \rightarrow B_e(\Manifold, \R^{n+1})
	\end{equation}	
makes $\operatorname{Emb}(\Manifold, \R^{n+1})$ the \emph{total space of the smooth principal fibration} with $\operatorname{Diff}(\Manifold)$ acting as the \emph{structure group} or \emph{standard fiber}, and $B_e^n$ being the \emph{base space}, which goes back to \cite{binz1981manifold}. 
As a reminder, a fiber bundle is a manifold, which locally looks like a product space $B\times F$, where $B$ corresponds to the base space, and $F$ corresponds to the standard fiber.
In our context, this means the pre-shape space $\operatorname{Emb}(\Manifold, \R^{n+1})$ is the collection of parameterized shapes, which locally looks like 'Shape'$\times$'Parameterization'.
However, this relationship holds only locally, and the global structure of the pre-shape space $\operatorname{Emb}(\Manifold, \R^{n+1})$ is much more complex.
The situation is graphically sketched in \cref{Fig_EmbSpace}.
\begin{figure}[h!]\centering
	\vspace{-4.5cm}
	\def\svgwidth{350.pt}
	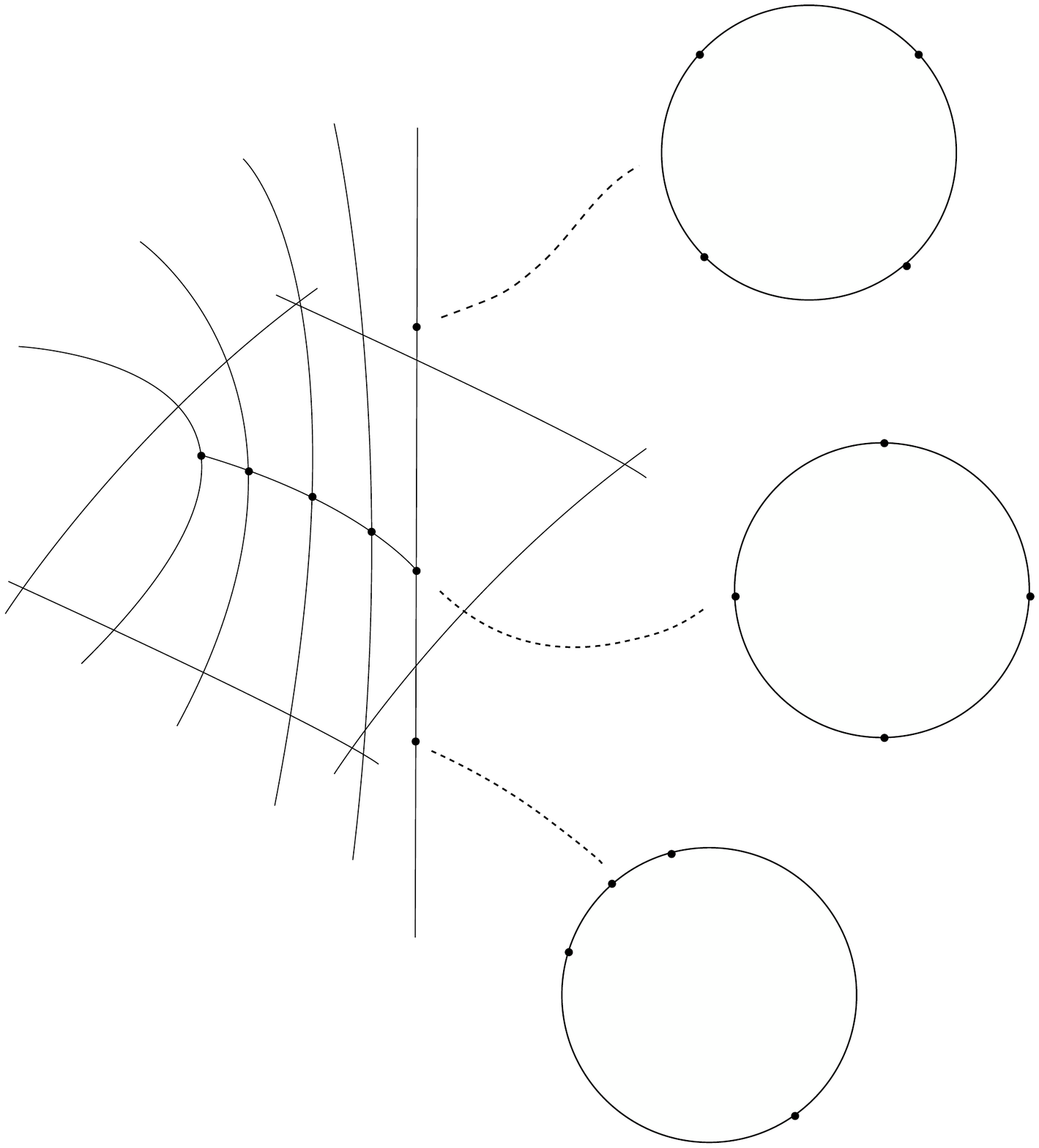
	\caption{\label{Fig_EmbSpace}Illustration of the pre-shape space $\operatorname{Emb}(\Manifold, \R^{n+1})$ for $\Manifold=S^1$. To illustrate the parameterization interpretation of fibers $\ProjectionCanonical(\ShapeEmbedding)$, the same four points are mapped from $\Manifold$ to $\ShapeEmbedding_i(\Manifold)$.}
\end{figure}
An application of the bundle projection $\ProjectionCanonical$ to a parameterized shape $\ShapeEmbedding\in\operatorname{Emb}(\Manifold, \R^{n+1})$ results in its unparameterized shape $\ProjectionCanonical(\ShapeEmbedding)\in B_e^n$ in the base space. 
Hence, we can view the fiber $\ProjectionCanonical(\ShapeEmbedding)$ as the collection of all parameterizations of the shape $\ShapeEmbedding(\Manifold)$.
It is important to avoid confusion of $\ShapeEmbedding(\Manifold)$ and $\ProjectionCanonical(\ShapeEmbedding)$, which are both called shapes.
The first interprets shapes as \emph{subsets} $\ShapeEmbedding(\Manifold)\subset\R^{n+1}$, the latter as  \emph{equivalence-classes}, i.e. 
\begin{equation}
	\ProjectionCanonical(\ShapeEmbedding) := \{\psi\in\operatorname{Emb}(\Manifold, \R^{n+1}):\; \exists \Diffeomorphism\in\operatorname{Diff}(\Manifold) \; s.t.\; \ShapeEmbedding=\psi\circ\Diffeomorphism \} \in B_e^n.
\end{equation}
The equivalence class interpretation is the collection of parameterizations corresponding to a certain shape in $\R^{n+1}$.

In order to formulate an analogue of shape calculus in $\operatorname{Emb}(\Manifold,\R^{n+1})$, we need to characterize the tangential bundles $T\operatorname{Emb}(\Manifold,\R^{n+1})$ and $TB_e^n$, as well as their relations.
For this, we make use of results by Michor and Kriegl \cite{kriegl1997convenient}. 

Since we assume $\Manifold$ to be compact, the respective tangent bundle of the pre-shape space is isomorphically given by
	\begin{equation}
		T_\ShapeEmbedding\operatorname{Emb}(\Manifold,\R^{n+1}) \cong C^\infty(\ShapeEmbedding(\Manifold), \R^{n+1}) \quad \forall \ShapeEmbedding\in\operatorname{Emb}(\Manifold, \R^{n+1}).
	\end{equation}
The fiber-bundle structure leads to a decomposition of the tangent bundle of the total space $T\operatorname{Emb}(\Manifold,\R^{n+1})$ into 
the so called \emph{vertical bundle}, defined as $\ker T\ProjectionCanonical \subset T\operatorname{Emb}(\Manifold,\R^{n+1})$, and the \emph{horizontal bundle}. 
Since we only deal with compact and orientable $n$-dimensional manifolds $\Manifold$, the existence of outer normal vector-fields $n$ on $\ShapeEmbedding(\Manifold)\subset \R^{n+1}$ is guaranteed.
In the following, let $\langle .,. \rangle_2$ denote the $L^2$-scalar product.
Thus we obtain
	\begin{equation}\label{Defi_DecompositionTangentEmb}
		T_\ShapeEmbedding\operatorname{Emb}(\Manifold,\R^{n+1}) \cong \TangentSpaceShape_{\ShapeEmbedding(\Manifold)} \oplus \NormalSpaceShape_{\ShapeEmbedding(\Manifold)} \quad \forall \ShapeEmbedding \in \operatorname{Emb}(\Manifold,\R^{n+1}),
	\end{equation}
where 
	\begin{equation}
		\TangentSpaceShape_{\ShapeEmbedding(\Manifold)} := \{h \in C^\infty(\ShapeEmbedding(\Manifold), \R^{n+1}): \langle h, n \rangle_{2} = 0 \; \text{ on } \ShapeEmbedding(\Manifold)\}
	\end{equation} 
is the \emph{space of tangential vector fields on $\ShapeEmbedding(\Manifold)$} in the ambient space $\R^{n+1}$ and 
	\begin{equation}
		\NormalSpaceShape_{\ShapeEmbedding(\Manifold)} := \{h\in C^\infty(\ShapeEmbedding(\Manifold), \R^{n+1}): h = \alpha\cdot n,\; \alpha \in C^\infty(\ShapeEmbedding(\Manifold), \R) \} 
	\end{equation}
being the \emph{space of normal vector fields on $\ShapeEmbedding(\Manifold)$}.
The tangential fields are parts of the vertical bundle, whereas the normal fields constitute the horizontal bundle part.
This also gives the well-known characterization of the tangential bundle of the classical shape space $B_e^n$ via normal vector fields, i.e.
	\begin{equation}\label{Defi_TangentBundleBe}
		T_{\ProjectionCanonical(\ShapeEmbedding)} B_e^n \cong \NormalSpaceShape_{\ShapeEmbedding(\Manifold)} \cong C^\infty(\ShapeEmbedding(\Manifold), \R)  \quad \forall \ProjectionCanonical(\ShapeEmbedding) \in B_e^n.
	\end{equation}
As previously, we also visualize the situation for tangential bundles in pictures \cref{Fig_TangBeSpace} and \cref{Fig_TangEmbSpace}.

\begin{figure}[h!]\centering
	\vspace{-8.25cm}
	\def\svgwidth{350.pt}
	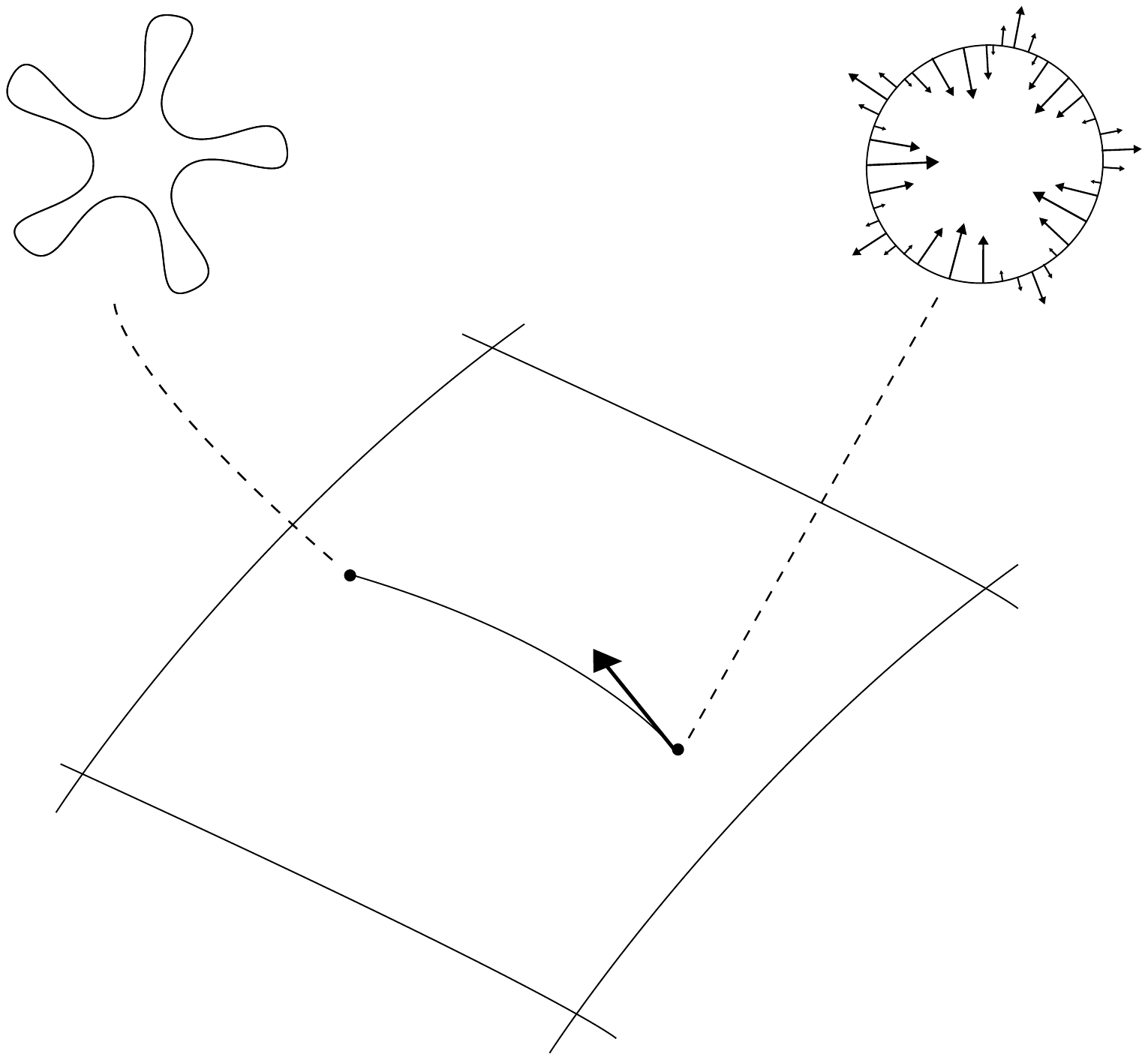
	\caption{	\label{Fig_TangBeSpace}Illustration of a shape tangential vector from  $TB_e^n$ for $\Manifold=S^1$.}
\end{figure}
\begin{figure}[h!]\centering
	\vspace{-4.85cm}
	\def\svgwidth{350.pt}
	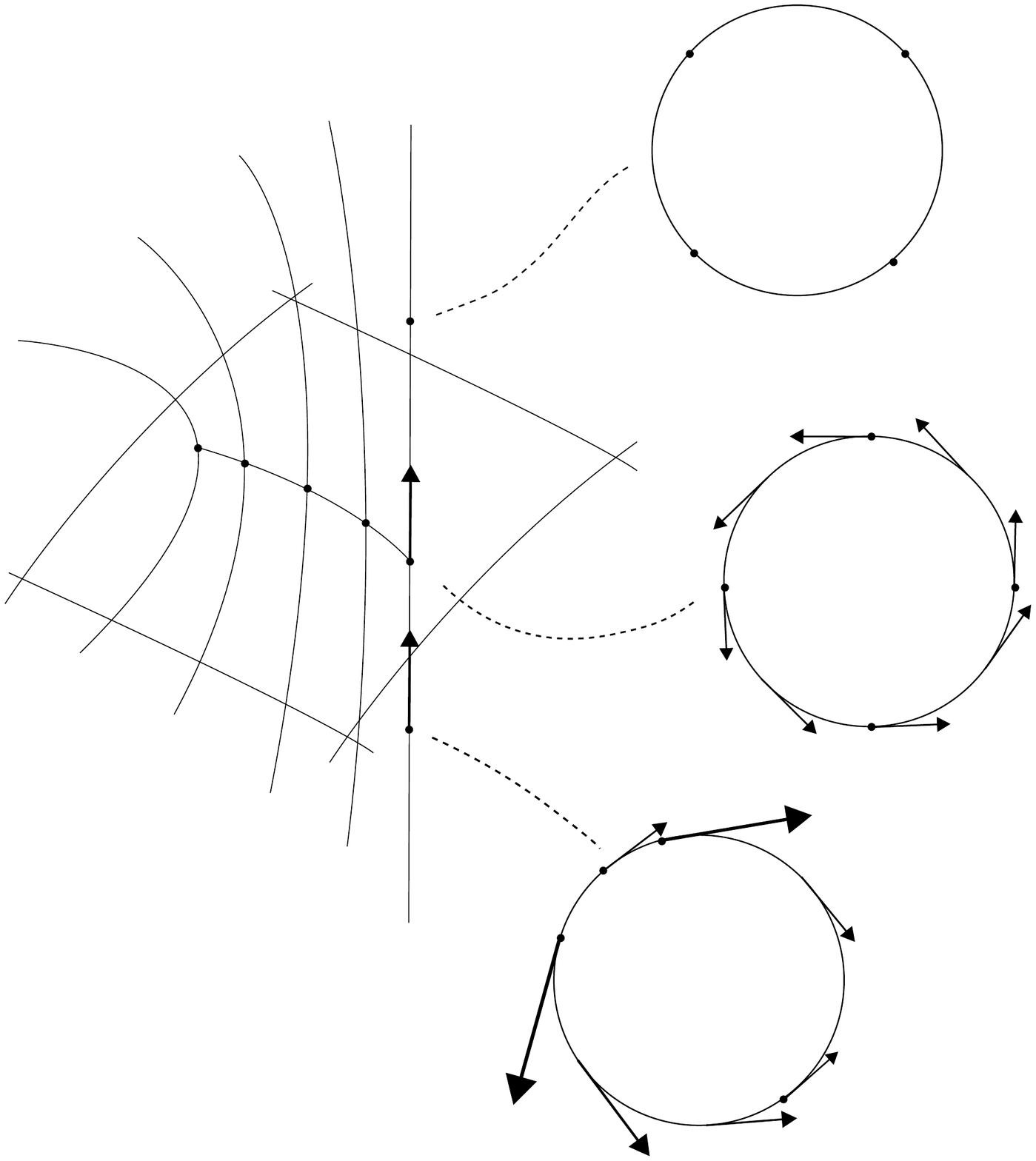
	\caption{	\label{Fig_TangEmbSpace}Illustration of vectors from  $T\operatorname{Emb}(\Manifold, \R^{n+1})$ with pure tangential/vertical components for $\Manifold=S^1$. Note that four points are added to illustrate the parameterization interpretation of fibers $\ProjectionCanonical(\ShapeEmbedding)$.}
\end{figure}
\paragraph[Pre-Shape Calculus]{Pre-Shape Calculus}
Next, we introduce a suitable notion of objective functionals. 
We are inspired by \cite[Ch. 4.3.1]{Delfour-Zolesio-2001}, where shape functionals $\TargetShp$ are functions on a set of admissible shapes $\ShapeAdmissibleSet$, which are considered to be a subset of the power set $\ShapeAdmissibleSet\subseteq\mathcal{P}(\R^{n+1})$. 
This is a set-theoretic approach, because the power set $\mathcal{P}(\R^{n+1})$ is the set of all subsets of $\R^{n+1}$.
Since we can canonically associate every equivalence class  $\ProjectionCanonical(\ShapeEmbedding)\in B_e^n$ with its set  $\ShapeEmbedding(\Manifold)\subset\R^{n+1}$, we get the following definition for the special set of admissible shapes $B_e^n$.
\begin{definition}[Shape and Pre-Shape Functionals]\label{Defi_ShapeFunctionals}
		Let $\Manifold$ be an $n$-dimensional, orientable, path-connected and compact $C^\infty$-manifold. 			
		Consider the shape space $B_e^n$ as defined in \cref{Defi_Be} and the space of embeddings $\operatorname{Emb}(\Manifold, \R^{n+1})$. 
		
		Then a function
			\begin{equation}\label{Defi_ShapeFunctional}
				\TargetShp: B_e^n \rightarrow \R
			\end{equation}
		is called \emph{shape functional}, and a function
			\begin{equation}\label{Defi_PreShapeFunctional}
				\TargetPrShp: \operatorname{Emb}(\Manifold, \R^{n+1}) \rightarrow \R
			\end{equation}
		is called \emph{pre-shape functional}.
\end{definition}
The nomenclature pre-shape functional for functions as in \cref{Defi_PreShapeFunctional} is motivated by regarding  $\operatorname{Emb}(\Manifold,\R^{n+1})$ as a pre-shape space as done by Michor et. al. in \cite[Ch. 1.1]{bauer2014overview}. 
Since optimization is classically taking place in shape spaces as opposed to pre-shape spaces, we will highlight some of their correspondences and relations. 
The following definition is motivated by the construction of the shape space $B_e^n$ in \cref{Defi_Be}.
	\begin{definition}[Shape Functionality]\label{Defi_ShapeFunctionalityPreShFctnl}
		Let $\TargetPrShp$ be a pre-shape functional and let $\ShapeEmbedding\in \operatorname{Emb}(\Manifold,\R^{n+1})$. 
		We say $\TargetPrShp$ has \emph{shape functionality in $\ShapeEmbedding$} if it is consistent with the fiber projection, i.e.
			\begin{equation}
				\TargetPrShp(\ShapeEmbedding\circ \Diffeomorphism) = \TargetPrShp(\ShapeEmbedding) \quad \forall \Diffeomorphism \in \operatorname{Diff}(\Manifold).
			\end{equation}
		If $\TargetPrShp$ has shape functionality for all $\ShapeEmbedding\in \operatorname{Emb}(\Manifold,\R^{n+1})$, we say \emph{$\TargetPrShp$ has shape functionality}.
	\end{definition}

In order to give optimality criteria for the pre-shape optimization problems and to formulate derivative based optimization algorithms, we need to introduce a shape derivative analogue for $\operatorname{Emb}(\Manifold, \R^{n+1})$.
Also, it is desirable that the analogue is compatible with the classical Eulerian derivative, 
as for example found in \cite[Ch. 4.3.2]{Delfour-Zolesio-2001} or \cite[Ch. 2.1]{schulz2015Steklov}. 
This motivates us to proceed in the fashion of classical shape optimization by defining a \emph{pre-shape derivatives} based on families of deformations perturbing the image space. 
We show their relation to classical shape derivatives, and then give a structure theorem similar to the Hadamard-Zol\'{e}sio structure theorem (cf. \cite[Ch. 9, Thrm. 3.6] {Delfour-Zolesio-2001}). 
Shape calculus or sensitivity analysis of classical shape optimization (cf. \cite[Ch. 3]{HaslingerMakinen}, \cite{berggren2010unified}) will carry over to pre-shape spaces naturally.
\begin{remark}[Validity of Pre-Shape Theory for different Regularities of Shapes]\label{Remark_DifferentReguSettings}
	We want to remind the reader, that the choice of $C^\infty$-regularity for pre-shapes in $\operatorname{Emb}(\Manifold, \R^{n+1})$ is not necessary to introduce the concepts of this section, but merely serves as an exemplary situation.
	It is clear, that the same definitions apply for embeddings $\ShapeEmbedding$ of Sobolev- or H\"older-regularity.
	In these cases test functions and directions $V$ of course need to have according regularity.
\end{remark}
\begin{definition}[Perturbation of Identity and Pre-shape Derivatives]\label{Defi_PreShapeDerivAndPertuOfId}
	Let $\TargetPrShp$ be a pre-shape functional (not necessarily having shape functionality), $\ShapeEmbedding\in\operatorname{Emb}(\Manifold, \R^{n+1})$ and $V\in C^{\infty}(\R^{n+1}, \R^{n+1})$. 
	Then the family of functions
	\begin{equation}\label{Defi_PertuOfIdPreShapes}
		\ShapeEmbedding_t := \ShapeEmbedding + t\cdot V\circ\ShapeEmbedding
	\end{equation}
	is called \emph{perturbation of identity of $\ShapeEmbedding$ in direction $V$} for $t\in[0,\tau)$ and some $\tau >0$. 
	The limit
	\begin{equation}\label{Defi_PreShapeDeriv}
		\PrShpDeriv\TargetPrShp(\ShapeEmbedding)[V] := \lim_{t\downarrow0} \frac{\TargetPrShp(\ShapeEmbedding_t) - \TargetPrShp(\ShapeEmbedding)}{t}
	\end{equation}
	is called \emph{pre-shape derivative} for $\TargetPrShp$ at  $\ShapeEmbedding \in \operatorname{Emb}(\Manifold, \R^{n+1})$ in direction $V$, if it exists and is linear and bounded in $V$.
\end{definition}

The perturbation of identity for shapes at $\Gamma_0\subset \R^{n+1}$ in direction $V\in C^\infty(\R^{n+1},\R^{n+1})$ is defined by
\begin{equation}
	\Gamma_t := \{x_0 + t\cdot V(x_0): x_0 \in \Gamma_0  \}.
\end{equation}
Notice that this is a set, in contrast to the perturbation of identity for pre-shapes \cref{Defi_PertuOfIdPreShapes} which is a function in $\operatorname{Emb}(\Manifold, \R^{n+1})$.
Shape derivatives of a shape functional $\TargetShp$ are given by
\begin{equation}
	\ShpDeriv \TargetShp(\Gamma_0)[V] := \underset{t\downarrow0}{\lim} \frac{\TargetShp(\Gamma_t) -\TargetShp(\Gamma_0)}{t}.
\end{equation}
The difference quotients defining pre-shape and shape derivatives use completely different objects, therefore their difference is significant.
Their relationship is explored in \cref{Prop_ShapeDiffImpliesPreShapeDiff} and \cref{Theorem_PreShapeHadamard}.

The next proposition shows a result relating shape differentiability of classical shape optimization and pre-shape derivatives.

\begin{proposition}[Shape Differentiability implies Pre-Shape Differentiability]
	\label{Prop_ShapeDiffImpliesPreShapeDiff}
	Consider a shape functional $\TargetShp: B_e^n \rightarrow \R$. 
	Then it has a canonical extension to a pre-shape functional
	\begin{equation}\label{Eq_CanonicalExtensionToPreShp}
		\TargetPrShp: \operatorname {Emb}(\Manifold, \R^{n+1}) \rightarrow \R,\; \ShapeEmbedding \mapsto \TargetShp(\ProjectionCanonical(\ShapeEmbedding)),
	\end{equation} 
	where $\ProjectionCanonical$ is the bundle projection as in \cref{Defi_FiberBundleEmb}.
	Further, there is a one-to-one correspondence of shape functionals $\TargetShp$ and pre-shape functionals $\TargetPrShp$ with the property of shape functionality.
	Additionally, if $\TargetShp$ is shape differentiable in the classical sense, then its extension $\TargetPrShp$ is pre-shape differentiable.
\end{proposition}
\begin{proof}
	One-to-one correspondence of pre-shape functionals with the property of shape functionality as in \cref{Defi_ShapeFunctionalityPreShFctnl} and classical shape functionals as in \cref{Defi_ShapeFunctional} is clear. 
	On the one hand, every canonical extension \cref{Eq_CanonicalExtensionToPreShp} of a classical shape functional $\TargetShp$ has shape functionality. 
	On the other hand, every pre-shape functional $\tilde{\TargetPrShp}$ having shape functionality gives rise to a shape functional $\TargetShp$ fulfilling \cref{Eq_CanonicalExtensionToPreShp}, as every fiber $\ProjectionCanonical(\ShapeEmbedding)$ is the orbit of a $\ShapeEmbedding$ by $\operatorname{Diff}(\Manifold)$
	acting from the right. 
	
	The pre-shape differentiability assertion in \cref{Prop_ShapeDiffImpliesPreShapeDiff} holds, since $\TargetShp\circ\ProjectionCanonical$ extends the values of $\TargetShp$ constantly onto the fibers of $\operatorname{Emb}(\Manifold, \R^{n+1})$. 
	For a fixed $\ShapeEmbedding\in \operatorname{Emb}(\Manifold, \R^{n+1})$, a case analysis for directions $V\in C^\infty(\R^{n+1}, \R^{n+1})$ being either tangential or normal at $\ShapeEmbedding(\Manifold)$ can be made.
	In case of horizontal, i.e. normal, directions $V$ we recover the classical shape derivative $\ShpDeriv\TargetShp$.
	On the other hand, the pre-shape derivative in vertical directions can be represented as a differential via curves on $\ShapeEmbedding(\Manifold)$, which combined with $\TargetShp\circ\ProjectionCanonical$ extending $\TargetShp$ constantly on fibers gives a vanishing pre-shape derivative.
	Linearity and boundedness of $\PrShpDeriv(\TargetShp\circ\ProjectionCanonical)(\ShapeEmbedding)[V]$ in $V$ are easy to see due to its vanishing for tangential components of $V$ together with linearity and boundedness of the classical shape derivative $\ShpDeriv\TargetShp$ by assumption.
\end{proof}

We can now situate classical shape optimization problems in the context of optimization in pre-shape spaces $\operatorname{Emb}(\Manifold, \R^{n+1})$ for suitable manifolds $\Manifold$. 
But first, we observe that a unique solution $\ShapeEmbedding(\Manifold)$ of a shape optimization problem has multiple parameterizations in general. 
For shape optimization problems posed in the pre-shape space $\operatorname{Emb}(\Manifold,\R^{n+1})$, this leads to non-uniqueness of solutions, which might at first seem like a disadvantage.
However, due to non-uniqueness up to elements in the solution fiber $\ProjectionCanonical(\ShapeEmbedding)$, it is possible to demand additional properties for the pre-shape solution.
This gives several opportunities to enhance numerical shape optimization routines, while at the same time narrowing down the amount of non-uniqueness of pre-shape solutions to a reasonable level. 
We will exploit this in upcoming works, such as \cite{luft2020pre2}.
For example, increasing mesh quality while not changing the shape at hand can be viewed as a condition posed on a shape optimization problem selecting a pre-shape in a given fiber.
Proposition \ref{Prop_ShapeDiffImpliesPreShapeDiff} also offers a possibility to transfer results concerning shape differentiability of classical shape functionals to the pre-shape setting without the need for new proofs. 
In particular, existence of stationary points in $B_e^n$ is carried over to $\operatorname{Emb}(\Manifold, \R^{n+1})$ as existence of stationary fibers. 
Hence \cref{Prop_ShapeDiffImpliesPreShapeDiff} shows that  pre-shape optimization is in some sense a canonical  generalization of classical shape optimization.

The definition of material- and shape derivatives found in classical shape optimization and structural sensitivity analysis literature (cf. \cite[Definition 1, Definition 2]{berggren2010unified}, \cite[Ch. 3.3.1]{HaslingerMakinen}) possesses useful properties for practical applications. 
In particular, through the use of material derivatives, it is often straightforward to derive expressions for shape derivatives of integral quantities.
We proceed by extending the notion of material derivatives from the classical context to the pre-shape calculus framework to harness these practical benefits.

\begin{definition}[Pre-Shape Material Derivative]\label{Defi_PreShapeMaterialDerivDefi}
	Consider a family of functions $\{f_\ShapeEmbedding: \R^{n+1} \rightarrow \R\}_{ \ShapeEmbedding\in \operatorname{Emb}(\Manifold, \R^{n+1})}$.
	For a direction $V\in C^\infty(\R^{n+1},\R^{n+1})$, we define the \emph{pre-shape material derivative} in $x_0\in \R^{n+1}$ by 
	\begin{equation}\label{Defi_PreShapeMaterialDeriv}
		\PrShpDeriv_mf (\ShapeEmbedding)[V](x_0) := \frac{\diff}{\diff t}_{\vert t=0} f_{\ShapeEmbedding_t}(x_t),
	\end{equation}
	if the limit exists. 
	Here, $\ShapeEmbedding_t$ is the perturbation of identity for pre-shapes (cf. \cref{Defi_PertuOfIdPreShapes}) and $x_t = x_0 + t\cdot V(x_0)$ is a perturbed point.
\end{definition}
The careful reader might notice the similarity of classical shape and pre-shape material derivatives.
However, the main difference is a possible dependence of functions $f$ on parameterizations of shapes/domains they are defined for.
Still, both notions coincide if the pre-shape functional has shape functionality, as we will see in \cref{Cor_MaterialDeriv} coming from the main structure \cref{Theorem_PreShapeHadamard}.

The definition of the material derivative can be generalized to functions and domains of weaker regularity, such as Sobolev functions and open subset $\Omega\subset \R^{n+1}$ with Lipschitz boundaries.
A necessity for this comes from the fact, that state solutions stemming from PDE constrained shape optimization problems need a well-defined material derivative for sensitivity analysis to be convenient.
This is done in the same way as with the classical shape material derivative (cf. \cite[p. 111]{HaslingerMakinen}).

It is important to notice that the family $f_\ShapeEmbedding$ can be seen as a function
$f: \operatorname{Emb}(\Manifold, \R^{n+1})\times \R^{n+1} \rightarrow \R$.
In the first component, the perturbation of identity for pre-shapes comes into play, which differs from the classical shape-material derivative. 
This leads to the following decomposition of the material derivative, which is similar to classical shape calculus, e.g. given by Haslinger and M\"akinen in \cite[p. 111, (3.39)]{HaslingerMakinen},
\begin{equation}\label{Eq_MaterialDerivDecompo}
	\PrShpDeriv_mf(\ShapeEmbedding)[V] = \PrShpDeriv f(\ShapeEmbedding)[V] + \nabla f_\ShapeEmbedding^T V.
\end{equation}
In the following we give a characterization of the pre-shape derivative in the style of the Hadamard-Zol\'{e}sio structure theorem as found in \cite[Ch. 9, Thm. 3.6]{Delfour-Zolesio-2001}.
For explanations concerning the use of distributions, as we do in the following, the reader can consult \cite[Remark 6.2, Defi. 6.22, Defi. 6.34, Thrm. 7.10, Ex. 7.12]{rudin1991functional}.
\begin{theorem}[Structure Theorem for Pre-Shape Derivatives]\label{Theorem_PreShapeHadamard}$\;$\newline
	Let $\TargetPrShp: \operatorname{Emb}(\Manifold, \R^{n+1}) \rightarrow \R$ be a pre-shape differentiable pre-shape functional (not necessarily having shape functionality)  and let $\ShapeEmbedding\in \operatorname{Emb}(\Manifold, \R^{n+1})$. Denote by $n_{\ShapeEmbedding(\Manifold)}$ the outer normal vector field of a shape  $\ShapeEmbedding(\Manifold)$ for a $\ShapeEmbedding\in \operatorname{Emb}(\Manifold, \R^{n+1})$.
	
	Then the following holds:
	\begin{itemize}
		\item[(i)] 
		The support of $\PrShpDeriv\TargetPrShp(\ShapeEmbedding)$ is given by
		\begin{equation}\label{Thrm_SuppPreShDir}
		\operatorname{supp} \PrShpDeriv\TargetPrShp(\ShapeEmbedding) \subseteq  \{V\in C^\infty(\R^{n+1}, \R^{n+1}): \ShapeEmbedding(\Manifold) \cap \operatorname{supp} V \neq \varnothing \}.
		\end{equation}
		\item[(ii)] There exist continuous linear functionals $g^{\TangentSpaceShape}: C^\infty(\R^{n+1}, \R^{n+1}) \rightarrow \R$ and $g^{\NormalSpaceShape}: C^\infty(\R^{n+1}, \R^{n+1}) \rightarrow \R$ depending on $\ShapeEmbedding$, which are tempered distributions when restricted to $C_c^\infty(\R^{n+1}, \R^{n+1})$ (cf. \cite[Ch. 6.1]{rudin1991functional} for definitions), with support on $\ShapeEmbedding(\Manifold)$ such that
		\begin{equation}\label{Thrm_HadamardSplittingPreShDir}
			\PrShpDeriv\TargetPrShp(\ShapeEmbedding)[V] = \big\langle g^{\NormalSpaceShape}, V \big\rangle + \big\langle g^{\TangentSpaceShape}, V \big\rangle \qquad \forall V\in C^\infty(\R^{n+1},\R^{n+1})
		\end{equation}
		with
		\begin{equation}\label{Eq_SupportNormalDistr}
			\operatorname{supp} g^{\NormalSpaceShape} \subseteq \operatorname{supp}\PrShpDeriv\TargetPrShp(\ShapeEmbedding)\cap \{V\in C_c^\infty(\R^{n+1},\R^{n+1}): \operatorname{Tr}_{\vert\ShapeEmbedding(\Manifold)}[V] \in \NormalSpaceShape_{\ShapeEmbedding(\Manifold)} \}
		\end{equation} 
		and
		\begin{equation}\label{Eq_SupportTangDistr}
			\operatorname{supp} g^{\TangentSpaceShape} \subseteq \operatorname{supp}\PrShpDeriv\TargetPrShp(\ShapeEmbedding)\cap \{V\in C_c^\infty(\R^{n+1},\R^{n+1}): \operatorname{Tr}_{\vert\ShapeEmbedding(\Manifold)}[V] \in \TangentSpaceShape_{\ShapeEmbedding(\Manifold)} \},
		\end{equation} 
		where $\operatorname{Tr}_{\vert\ShapeEmbedding(\Manifold)}: C^\infty(\R^{n+1}, \R^{n+1}) \rightarrow C^\infty(\ShapeEmbedding(\Manifold), \R^{n+1})$ is the trace operator and $n$ the outer unit normal vector field on $\ShapeEmbedding(\Manifold)$.
		
		\item[(iii)] If $\TargetPrShp$ has shape functionality, then for all $\ShapeEmbedding \in \operatorname{Emb}(\Manifold, \R^{n+1})$, we have  $g^{\TangentSpaceShape} = 0$ and 
		\begin{equation}\label{Thrm_HadaShapeFunct1}
			\PrShpDeriv\TargetPrShp(\ShapeEmbedding)[V] = \ShpDeriv\TargetShp(\ProjectionCanonical(\ShapeEmbedding))[V] \qquad \forall V \in C^\infty(\R^{n+1}, \R^{n+1}),
		\end{equation}
		where $\TargetShp: B_e^n \rightarrow \R$ is the natural shape functional corresponding to $\TargetPrShp$ by $\TargetShp\circ\ProjectionCanonical = \TargetPrShp$. 
		In particular $g^{\NormalSpaceShape}$ corresponds to the distribution of the classical Hadamard-Zol\'{e}sio structure theorem.
	\end{itemize}
\end{theorem}
\begin{proof}
	Since we are in a different situation than the classical Hadamard-Zol\'{e}sio structure theorem \cite[Ch. 9, Thm. 3.6] {Delfour-Zolesio-2001} for (i) and $(ii)$, we give proofs for these ourselves. 
	
	For (i), let $V\in C^\infty(\R^{n+1}, \R^{n+1})$ be such that $\ShapeEmbedding(\Manifold) \cap \operatorname{supp} V = \varnothing$. 
	Consider the perturbation of identity $\ShapeEmbedding_t$ for $V$ of $\ShapeEmbedding$ as in \cref{Defi_PertuOfIdPreShapes}.
	By $\ShapeEmbedding(\Manifold) \cap \operatorname{supp} V = \varnothing$ we have $V\circ\ShapeEmbedding = 0$, resulting in $\ShapeEmbedding_t = \ShapeEmbedding$ being constant in $t$.
	This yields $\PrShpDeriv\TargetPrShp(\ShapeEmbedding)[V] = 0$ by \cref{Defi_PreShapeDeriv}, which immediately gives us (i).
	
	For (ii), the proof, in some extent, follows analogous reasoning as in \cite[Ch. 9.3.4, Cor. 1]{Delfour-Zolesio-2001}, where Banach spaces $C^k(\R^{n+1}, \R^{n+1})$ are considered. 
	It is clear that $\PrShpDeriv\TargetPrShp(\ShapeEmbedding): C^\infty(\R^{n+1}, \R^{n+1}) \rightarrow \R$ defines a linear functional with compact support as in \cref{Thrm_SuppPreShDir} (cf. \cite[Def. 6.22]{rudin1991functional} for definition of supports of distributions). 
	In addition, we can use that it is contained in the Schwartz space $C^\infty_c(\R^{n+1},\R^{n+1})$ due to compactness of $\ShapeEmbedding(\Manifold)$.
	This gives us the tempered distribution property.
	Then $g^{\TangentSpaceShape}$ and $g^{\NormalSpaceShape}$ are defined by restriction to the vertical and horizontal part of $V$, recurring on decomposition of the tangent bundle in horizontal and vertical components \cref{Defi_DecompositionTangentEmb}, giving us (ii).

	For (iii), let $\TargetPrShp$ have shape functionality and let $V\in C^\infty(\R^{n+1}, \R^{n+1})$. 
	On $\ShapeEmbedding(\Manifold)$, we can decompose $V$  into normal and tangential components. 
	For the tangential part, we can follow analogous arguments as in the proof of \cref{Prop_ShapeDiffImpliesPreShapeDiff}, giving us a curve $\ShapeEmbedding_t$ in the fiber $\ProjectionCanonical(\ShapeEmbedding)$ generating the pre-shape derivative as a differential at a given $\ShapeEmbedding$. 
	As $\ShapeEmbedding_t$ is running on the fiber of $\ShapeEmbedding$ and $\TargetPrShp$ has shape functionality (cf. \cref{Defi_ShapeFunctionalityPreShFctnl}) in $\ShapeEmbedding$ by assumption, $\TargetPrShp(\ShapeEmbedding_t)$ is constant for all $t$, rendering $g^\TangentSpaceShape=0$ by \cref{Thrm_HadamardSplittingPreShDir}. 
	Further, by \cref{Thrm_HadamardSplittingPreShDir} the pre-shape derivative $\PrShpDeriv\TargetPrShp(\ShapeEmbedding)[V]$ reduces to $g^{\NormalSpaceShape}$ acting on normal directions.
	With \cref{Prop_ShapeDiffImpliesPreShapeDiff}, shape functionality of $\TargetPrShp$ leads to a well defined shape functional $\TargetShp: B_e^n \rightarrow \R$ with $\TargetShp\circ \ProjectionCanonical = \TargetPrShp$. 
	As the tangential part of $V$ has no impact on $\PrShpDeriv\TargetPrShp(\ShapeEmbedding)[V]$, we can find a horizontal curve $\ShapeEmbedding_t$ generating $\PrShpDeriv\TargetPrShp(\ShapeEmbedding)[V]$. 
	The representative $\ShapeEmbedding_t$ either creates a trivial curve $\ProjectionCanonical(\ShapeEmbedding_t)$ in $B_e^n$, which leads to \cref{Thrm_HadaShapeFunct1} being $0$ on both sides, or a non-trivial curve $\ProjectionCanonical(\ShapeEmbedding_t)$ in $B_e^n$. 
	If $\ProjectionCanonical(\ShapeEmbedding_t)$ is non trivial, we have
	\begin{align}
		\ShpDeriv\TargetShp(\ProjectionCanonical(\ShapeEmbedding))[V] =  \frac{\diff}{\diff t}_{\vert t=0}   \TargetShp(\ProjectionCanonical(\ShapeEmbedding_t)) = \frac{\diff}{\diff t}_{\vert t=0}  \TargetPrShp(\ShapeEmbedding_t) =
		\PrShpDeriv\TargetPrShp(\ShapeEmbedding)[V]  
	\end{align}
	for the shape derivative $\ShpDeriv\TargetShp(\ProjectionCanonical(\ShapeEmbedding))$ and pre-shape derivative $\PrShpDeriv\TargetPrShp(\ShapeEmbedding)$, resulting in \cref{Thrm_HadaShapeFunct1}. 
	By association of $\ShapeEmbedding(\Manifold)$ with $\ProjectionCanonical(\ShapeEmbedding)\in B_e^n$ this also shows that $g^\NormalSpaceShape$ corresponds to the distribution in the classical Hadamard-Zol\'{e}sio structure theorem (cf. \cite[Ch. 9.3.4, Thrm. 3.6 and Cor. 1]{Delfour-Zolesio-2001}), giving us (iii).
\end{proof}

Structure \cref{Theorem_PreShapeHadamard} gives an intuitive way to understand the pre-shape derivative \cref{Defi_PreShapeDeriv} and the relation between shape functionals and pre-shape functionals.
Part (i) of \cref{Theorem_PreShapeHadamard} has the same meaning as in the classical Hadamard-Zol\'{e}sio structure theorem for shape derivatives, namely that deformations of the hold-all domain only influence the pre-shape functional if they deform the (pre-) shape $\ShapeEmbedding(\Manifold)$. 

The difference to classical shape derivatives is illustrated in equation \cref{Thrm_HadamardSplittingPreShDir}, where the effect of deformations on the objective is split into normal and tangential components. 
The normal part $g^{\NormalSpaceShape}$ can be understood as the shape optimization part of $\PrShpDeriv\TargetPrShp$, i.e. $\TargetPrShp$ depending on the change of interface $\ShapeEmbedding(\Manifold)$.
This is also reflected by the structure of its support given in \cref{Eq_SupportNormalDistr}, which states that only normal directions $V\in\NormalSpaceShape_{\ShapeEmbedding(\Manifold)}$ deforming $\ShapeEmbedding(\Manifold)$ have an effect on $g^{\NormalSpaceShape}$.
On the other hand, $g^{\TangentSpaceShape}$ is interpretable as the part of $\PrShpDeriv\TargetPrShp$ being sensitive to reparameterization of the shape $\ShapeEmbedding(\Manifold)$, which is shown by the structure of its support in \cref{Eq_SupportTangDistr}, where only tangential vector field $V\in\TangentSpaceShape_{\ShapeEmbedding(\Manifold)}$ play a role.
In classical shape optimization, tangential vectors are always in the nullspace of the shape derivative.
But in the more general pre-shape case both components discussed can have non-trivial effects.

This is also reflected by \cref{Theorem_PreShapeHadamard} (iii), stating that pre-shape functionals having shape functionality have vanishing tangential part of the pre-shape derivative $\PrShpDeriv\TargetPrShp$, meaning that they are only supported by normal components of the deformation field $V\in C^\infty(\R^{n+1}, \R^{n+1})$ just as in classical shape optimization theory.
On the other hand, if 'shape derivatives' are not found to vanish in tangential directions, the 'shape functional' at hand is actually a true pre-shape functional. 
This is the case for mesh optimization techniques we will introduce.

Also, for shapes $\ShapeEmbedding(\Manifold)\subset\R^{n+1}$ being bounded and topologically closed $C^{k+1}$-submanifolds of $\R^{n+1}$ with non-empty interior, the classical Hadamard-Zol\'{e}sio structure theorem was generalized in \cite[Thrm. 5.5]{sturm2016structure}.
In the special case of the objective $\TargetPrShp$ having shape functionality, i.e. vanishing tangential component of the pre-shape derivative as by \cref{Theorem_PreShapeHadamard}, the results of \cite[Corollary 4.2]{sturm2016structure} coincide with results in \cref{Theorem_PreShapeHadamard}.

Before we come to some exemplary pre-shape derivatives and their  decompositions, we formulate a simple corollary, which connects the classical material derivatives to their pre-shape versions.
\begin{corollary}[Decomposition for Pre-Shape Material Derivatives]\label{Cor_MaterialDeriv}
	Let $f: \operatorname{Emb}(\Manifold, \R^{n+1})\times \R^{n+1} \rightarrow \R$ be pre-shape differentiable, $\ShapeEmbedding\in\operatorname{Emb}(\Manifold, \R^{n+1})$ and $V\in C^\infty(\R^{n+1}, \R^{n+1})$.
	Then the material derivative decomposes to
	\begin{equation}\label{Eq_PreShapeMaterialDerivDecompo}
		\PrShpDeriv_m f(\ShapeEmbedding)[V] 
		=
		\big\langle g^{\NormalSpaceShape}, V \big\rangle + \big\langle g^{\TangentSpaceShape}, V \big\rangle + \nabla f_\ShapeEmbedding^T V.
	\end{equation}
	In particular, if $f$ has shape functionality, then for the corresponding shape dependent function $\tilde{f}: B_e\times \R^{n+1} \rightarrow \R$ the relationship
	\begin{equation}\label{Eq_PreShpMaterialCoincClassical}
		\PrShpDeriv_m f (\ShapeEmbedding)[V] = \ShpDeriv_m \tilde{f} (\ProjectionCanonical(\ShapeEmbedding))[V] 
	\end{equation}
	holds for all $\ShapeEmbedding \in \operatorname{Emb}(\Manifold, \R^{n+1})$, which means the pre-shape and classical material derivative coincide.
\end{corollary}

\begin{proof}
	To get decomposition \cref{Eq_PreShapeMaterialDerivDecompo}, we simply use formula \cref{Eq_MaterialDerivDecompo} and apply decomposition \cref{Thrm_HadamardSplittingPreShDir} of the structure theorem to the occurring pre-shape derivatives for fixed $x_0 \in \R^{n+1}$.
	If $f$ has shape functionality, we can apply part (iii) of \cref{Theorem_PreShapeHadamard} to decomposition \cref{Eq_PreShapeMaterialDerivDecompo} and see that the pre-shape material derivative equals the classical material derivative $\ShpDeriv_m \tilde{f} (\ProjectionCanonical(\ShapeEmbedding))[V]$, 
	since $g^{\NormalSpaceShape}$ corresponds to the distribution from the classical Hadamard theorem.
\end{proof}

With the last part of this corollary, we can apply pre-shape material derivatives to shape differentiable functions, yielding the same results as the classical material derivative by association with pre-shape extensions via \cref{Prop_ShapeDiffImpliesPreShapeDiff}.

We have now seen that structure \cref{Theorem_PreShapeHadamard}, \cref{Prop_ShapeDiffImpliesPreShapeDiff} and  \cref{Cor_MaterialDeriv} guarantee validity of classical shape calculus formulae and results in the context of pre-shapes. 
Pre-shape calculus can be applied to objects from shape optimization if they are associated with their corresponding  pre-shape counterparts, leading to the same derivatives and thus optimization methods.
Even further, it is possible to apply pre-shape calculus to mixed shape and pre-shape problems, where the shape part is  treated just as if shape calculus was applied, with the key difference that a pre-shape component would otherwise be non-accessible.
In the following, we show some simple examples which are not accessible by classical shape calculus.

\paragraph{Example 1}\label{Ex_PreShpFirstEx}
	For a target pre-shape $\tilde{\ShapeEmbedding}\in \operatorname{Emb}(S^1, \R^2)$, let us define a pre-shape optimization problem by 
	\begin{equation}\label{Eq_ExampleToyPreShape}
		\underset{\ShapeEmbedding\in \operatorname{Emb}(S^1, \R^2)}{\min} \; \frac{1}{2}\int_{S^1} \vert \ShapeEmbedding - \tilde{\ShapeEmbedding} \vert ^2\;\diff s \; =: \TargetPrShp(\ShapeEmbedding).
	\end{equation}
	The pre-shape functional $\TargetPrShp$ measures the difference of a target $\tilde{\ShapeEmbedding}$ to another parameterized shape $\ShapeEmbedding$.
	
	Its pre-shape derivative can be calculated for directions $V\in C^\infty(\R^2,\R^2)$ by using elementary techniques
	\begin{align}\label{Eq_ExamplePreShapeGeneral}
	\begin{split}
		\PrShpDeriv\TargetPrShp(\ShapeEmbedding)[V] 
		=&
		\frac{\diff}{\diff t}_{\vert t = 0} \frac{1}{2}\int_{S^1} \vert \ShapeEmbedding_t - \tilde{\ShapeEmbedding} \vert ^2\;\diff s \\
		=&
		\frac{1}{2}\int_{S^1}\frac{\diff}{\diff t}_{\vert t = 0} \langle \ShapeEmbedding + t\cdot V\circ\ShapeEmbedding - \tilde{\ShapeEmbedding}, \ShapeEmbedding + t\cdot V\circ\ShapeEmbedding - \tilde{\ShapeEmbedding} \rangle \;\diff s \\
		=& 
		\int_{S^1} \langle \ShapeEmbedding - \tilde{\ShapeEmbedding}, V\circ\ShapeEmbedding \rangle \;\diff s.
	\end{split}
	\end{align}
	We can choose $S^1$ with canonical parameterization as a starting pre-shape or point of reference by considering
	\begin{equation}
		\ShapeEmbedding_{id}: S^1\subset \R^2 \rightarrow \R^2, 
		\begin{pmatrix}
		x_1 \\ x_2
		\end{pmatrix}
		\mapsto
		\begin{pmatrix}
		x_1 \\ x_2
		\end{pmatrix}.
	\end{equation}
	In order to formulate the Hadamard-Zol\'esio-type decomposition \cref{Thrm_HadamardSplittingPreShDir}, we need the outer normal vector field and an oriented tangential vector field, which for $S^1$ are given by
	\begin{equation}
		n: S^1 \rightarrow \R^2, 
		\begin{pmatrix}
		x_1 \\ x_2
		\end{pmatrix}
		\mapsto
		\begin{pmatrix}
		x_1 \\ x_2
		\end{pmatrix}
		, \quad 
		\TangentVector: S^1 \rightarrow \R^2,
		\begin{pmatrix}
		x_1 \\ x_2
		\end{pmatrix}
		\mapsto
		\begin{pmatrix}
		-x_2 \\ x_1
		\end{pmatrix}.
	\end{equation}
	Now we can examine the problem for several different parameterized target shapes $\tilde{\ShapeEmbedding}\in\operatorname{Emb}(S^1, \R^2)$.
	First, we can consider rescaling by a factor $\alpha\in (0,\infty)$, which lets $S^1$ contract or expand. 
	The target for this is given by
	\begin{equation}
		\tilde{\ShapeEmbedding}: S^1 \rightarrow \R^2, 		\begin{pmatrix}
		x_1 \\ x_2
		\end{pmatrix}
		\mapsto \alpha\cdot
		\begin{pmatrix}
		x_1 \\ x_2
		\end{pmatrix}.
	\end{equation}
	Using \cref{Eq_ExamplePreShapeGeneral}, the pre-shape derivative becomes
	\begin{equation}
		\PrShpDeriv\TargetPrShp(\ShapeEmbedding_{id})[V] 
		= 
		\int_{S^1} (1-\alpha) \cdot \langle n, V \rangle \;\diff s.
	\end{equation}
	This shows that rescaling of $S^1$ has vanishing parameterization part $g^{\mathcal{T}} \equiv 0$, whereas the remaining shape component $g^{\mathcal{N}}$ is in the style of the classical Hadamard-Zol\'esio representation given above.
	In particular, only vector fields $V$ acting in normal direction are supported.
	
	Next, let us consider a rotation of the circle. 
	For this, we let $\alpha\in [0,2\pi)$ and consider target rotations
	\begin{equation}
		\tilde{\ShapeEmbedding}: S^1 \rightarrow \R^2, 
		\begin{pmatrix}
		x_1 \\ x_2
		\end{pmatrix}
		\mapsto
		\Bigg(
		\begin{matrix}
		\cos(\alpha) &-\sin(\alpha) \\ \sin(\alpha) & \cos(\alpha)
		\end{matrix}
		\Bigg)
		\begin{pmatrix}
		x_1 \\ x_2
		\end{pmatrix}.
	\end{equation}
	Plugging this into \cref{Eq_ExamplePreShapeGeneral} and doing some reformulations, the according decomposition becomes
	\begin{equation}
		\PrShpDeriv\TargetPrShp(\ShapeEmbedding_{id})[V] 
		= 
		\int_{S^1} \big(1-\cos(\alpha)\big)\cdot \langle n, V\rangle \;\diff s 
		+ 
		\int_{S^1} \sin(\alpha)\cdot \langle \TangentVector, V\rangle \;\diff s.
	\end{equation}
	Here we see both components of the decomposition, the first corresponding to the normal $g^{\mathcal{N}}$, and the second corresponding to the tangential part $g^{\mathcal{T}}$.
	Notice that the normal component vanishes exactly for trivial rotations, whereas the tangential part also vanishes for the reflection at origin case $\alpha=\pi$.
	
	Finally, we can also translate $S^1$ by some fixed $z\in \R^2$, which gives a target
	\begin{equation}
		\tilde{\ShapeEmbedding}: S^1 \rightarrow \R^2, 
		\begin{pmatrix}
		x_1 \\ x_2
		\end{pmatrix}
		\mapsto
		\begin{pmatrix}
		x_1 \\ x_2
		\end{pmatrix}
		+ z.
	\end{equation}
	The decomposition of the pre-shape derivative becomes
	\begin{equation}
		\PrShpDeriv\TargetPrShp(\ShapeEmbedding_{id})[V]
		= 
		\int_{S^1} \langle n, z \rangle \cdot \langle n, V\rangle \;\diff s 
		+ 
		\int_{S^1}\langle \TangentVector, z\rangle \cdot \langle \TangentVector, V\rangle \;\diff s,
	\end{equation}
	where decomposition into $g^{\mathcal{N}}$ and $g^{\mathcal{T}}$ depend on normal and tangential components of $z$ on $S^1$.

Next, we give a summary of several useful pre-shape calculus formulae.
Due to \cref{Prop_ShapeDiffImpliesPreShapeDiff} and \cref{Cor_MaterialDeriv}, they are also true for shape derivatives and functionals.
\begin{corollary}[Pre-Shape Calculus Rules]\label{Cor_PreShapeCalc}
	Let $f,g: \operatorname{Emb}(\Manifold, \R^{n+1}) \times \R^{n+1} \rightarrow \R$ be pre-shape differentiable and differentiable in the second component, and let $h: \R \rightarrow \R$ be differentiable.
	Let $\Omega\subseteq \R^{n+1}$ be an open, bounded domain with Lipschitz boundary, $\Shape$ an $n$-dimensional $C^\infty$-submanifold of $\R^{n+1}$. 
	Consider $\ShapeEmbedding\in \operatorname{Emb}(\Manifold, \R^{n+1})$ and $V\in C^\infty(\R^{n+1}, \R^{n+1})$.
	Then the following set of rules applies for pre-shape and -material derivatives, including the special case of shape derivatives.
	\begin{spacing}{1.75}	
	\begin{enumerate}\label{Eq_MaterialShapeDer}
		\item[(i)] 
		$\PrShpDeriv_m f(\ShapeEmbedding)[V] 
		\hspace{1.175cm} = 
		\PrShpDeriv f(\ShapeEmbedding)[V] + \nabla f_\ShapeEmbedding^T V$ 
		\item[(ii)]
		$ \PrShpDeriv_m (f\cdot g)(\ShapeEmbedding)[V] 
		\hspace{0.475cm} =
		\PrShpDeriv_m f(\ShapeEmbedding)[V]\cdot g_\ShapeEmbedding + f_\ShapeEmbedding \cdot \PrShpDeriv_m g (\ShapeEmbedding)[V]$
		\item[(iii)] 
		$ \PrShpDeriv_m( h\circ f)(\ShapeEmbedding)[V] 
		\hspace{0.375cm} =
		Dh(f_\ShapeEmbedding)\PrShpDeriv_mf(\ShapeEmbedding)[V]$
		\item[(iv)] 
		$\PrShpDeriv_m(\int_{\Omega}f\;\diff x)(\ShapeEmbedding)[V] 
		\hspace{0.cm} =
		\int_{\Omega}\PrShpDeriv_m f(\ShapeEmbedding)[V] + \operatorname{div}(V)f_\ShapeEmbedding\;\diff x$
		\item[(v)]
		$	\PrShpDeriv_m(\int_{\Gamma}f\;\diff s)(\ShapeEmbedding)[V] 
		\hspace{.07cm} =
		\int_{\Gamma} \PrShpDeriv_mf(\ShapeEmbedding)[V] + \operatorname{div}_\Gamma(V)f_\ShapeEmbedding \;\diff s$
	\end{enumerate}
	\end{spacing}
	with $\operatorname{div}_\Gamma(V)$ being the tangential divergence of $V$ on $\ShapeEmbedding$, and $Dh$ being the total derivative of $h$.
\end{corollary}

\begin{proof}
	Let the assumptions stated above hold.
	Identity (i) was already discussed in \cref{Cor_MaterialDeriv}.
	
	The product- and chain-rule (ii) and (iii) are simple consequences of the definition of the material derivative \cref{Defi_PreShapeMaterialDeriv}.
	
	For (iv), the conditions for \cite[Thrm. 5.2.2]{hernot2018shape} apply by considering $f_{\ShapeEmbedding_t}(x_t)$ a function of $t\geq0$.
	Alternatively, since we assumed Lipschitz boundary for $\Omega$, the change of variable formula is applicable and the standard proof found in \cite[p. 112, Lemma 3.3]{HaslingerMakinen} can be used as well.
	
	The situation for (v) is more involved. 
	For this, we refer the reader to \cite[Thrm. 5.4.17]{hernot2018shape} or \cite[Ch.9.4, Thrm. 4.3]{Delfour-Zolesio-2001}.
\end{proof}

\begin{remark}[Weakening Assumptions for Pre-Shape Calculus]
	The formulae provided in \cref{Cor_PreShapeCalc} hold in far greater generality.
	
	In particular, the chain rule (iii) can be stated for Fr\'echet differentiable operators $h$ on Banach spaces of continuous functions with help of \cite[Ch.9, Thrm. 2.5]{Delfour-Zolesio-2001}.
	
	Also, formula (iv) for volume integrals can be stated for domains $\Omega$ which are merely measurable, and pre-shape differentiable families of class  $\{f_\ShapeEmbedding \in W^{1,1}(\R^{n+1},\R)\}_{\ShapeEmbedding\in\operatorname{Emb}(\Manifold, \R^{n+1})}$ with use of \cite[Thrm. 5.2.2]{hernot2018shape}.
	
	Finally, formula (v) for boundary integrals can be generalized to compact hypersurfaces $\Shape\subset\R^{n+1}$ of $C^1$-regularity and pre-shape differentiable families of class $\{f_\ShapeEmbedding \in W^{1,1}(\R^{n+1},\R)\}_{\ShapeEmbedding\in\operatorname{Emb}(\Manifold, \R^{n+1})}$ by use of \cite[Thrm. 5.4.17]{hernot2018shape}.
\end{remark}


\paragraph{Example 2: Pre-Shape Parameterization Tracking Problem}\label{Section_ShapeParaTrackingProblem}
With this problem class we introduce a non-trivial example for pre-shape optimization problems.
Its pre-shape derivative in fact is no classical shape derivative, and thus not tractable by shape calculus.
Later on, this problem class will enable us to optimize the overall mesh quality of discretizations and representations of shapes similar to deformations methods going back to \cite{liao1992new}. 
To be specific, the fiber bundle structure of $\operatorname{Emb}(\Manifold,\R^{n+1})$ offers the possibility to modify a given shape optimization problem in such a way, that the original solution is maintained, while at the same time optimization for the parameterization can take place. 
This gives rise to several different mesh regularization algorithms, and justifies the numerical optimization procedures we will establish in future works. 

Before we further elaborate on this, we introduce necessary vocabulary and notation to formulate our problem in  $\operatorname{Emb}(\Manifold, \R^{n+1})$ for submanifolds $\Manifold\subset \R^{n+1}$. 
First, we need the concept of \emph{local frames}, which are local orthonormal bases of tangential vectors on $\Manifold$ (cf. \cite[Ch. 8]{lee2013smooth}).
For an open subset $U\subseteq\Manifold$, a smooth local frame is a tuple of $\dim(M)$ tangential vector fields $\TangentVector := (\TangentVector_1,\dots,\TangentVector_n)$, such that for each $p\in U$ the tangential vectors $\TangentVector_i(p)\in T_p\Manifold$ are linearly independent.
If we have a Riemannian metric on $\Manifold$, then we can additionally demand $\TangentVector(p)=(\TangentVector_1(p),\dots,\TangentVector_n(p))$ to be orthonormal with respect to the Riemannian metric for all $p\in U$, thus calling the frame $(\TangentVector_1,\dots,\TangentVector_n)$ \emph{local orthonormal frame}.
Note that local orthonormal frames always exist, due to simple use of the Gram-Schmidt algorithm in tangential spaces (cf. \cite[Lemma 8.13]{lee2013smooth}).

To achieve a natural and numerically tractable formulation of a pre-shape parameterization tracking problem, we also need to introduce the covariant derivative of an embedding $\ShapeEmbedding\in \operatorname{Emb}(\Manifold, \R^{n+1})$. 
For this, we use similar ideas as in  \cite[Ch. 9.5.6]{Delfour-Zolesio-2001} and modify them to our situation using local orthonormal frames.
Given a $\ShapeEmbedding \in \operatorname{Emb}(\Manifold, \R^{n+1})$, let $\TangentVector: U \rightarrow (T\Manifold)^n$ be a smooth local orthonormal frame on $U\subseteq\Manifold$ and let $\TangentVector^\ShapeEmbedding: V \rightarrow (T\ShapeEmbedding(\Manifold))^n$ be a local orthonormal frame on $V\subseteq \ShapeEmbedding(\Manifold)$.
Without loss of generality, we can assume $V = \ShapeEmbedding(U)$, since we can choose $V = V\cap\ShapeEmbedding(U)$.
Then we define the local \emph{covariant derivative} representation for $\ShapeEmbedding$ under choice of frames $\TangentVector$ and $\TangentVector^{\ShapeEmbedding}$ by
\begin{equation}\label{Defi_CovariateDeriv}
D^\tau \ShapeEmbedding_{\vert U}(p) := 
\left(\begin{matrix}
\langle D\ShapeEmbedding \TangentVector_{1,p}, \TangentVector_{1,\ShapeEmbedding(p)}^{\ShapeEmbedding} \rangle & \dots & \langle D\ShapeEmbedding \TangentVector_{n,p}, \TangentVector_{1,\ShapeEmbedding(p)}^{\ShapeEmbedding} \rangle \\
\vdots & \ddots & \vdots \\
\langle D\ShapeEmbedding \TangentVector_{1,p}, \TangentVector_{n,\ShapeEmbedding(p)}^{\ShapeEmbedding} \rangle & \dots & \langle D\ShapeEmbedding \TangentVector_{n,p}, \TangentVector_{n,\ShapeEmbedding(p)}^{\ShapeEmbedding} \rangle
\end{matrix}\right),
\end{equation}
where $D\ShapeEmbedding$ is the Jacobian matrix of $\ShapeEmbedding$ and $\langle .,. \rangle$ the Euclidean scalar product of $\R^{n+1}$.
We want to make clear that the covariant derivative $D^\tau \ShapeEmbedding$ should not be mistaken for the tangential derivative $D_\Gamma\ShapeEmbedding$, which is given by (cf. \cite[Ch. 9.5.2]{Delfour-Zolesio-2001})
\begin{equation}
	D_\Gamma\ShapeEmbedding = D\ShapeEmbedding - D\ShapeEmbedding n n^T.
\end{equation}

Having numerical implementations in mind, we are also interested in the case of shapes with non-trivial boundaries.
However, when the boundary is non-trivial, we impose restrictions on the pre-shapes permitted.
Specifically, embeddings leaving the boundary invariant are sufficient, i.e. 
\begin{equation}\label{Eq_PreShapeSpaceNonTrivBoundary}
\operatorname{Emb}_{\partial\Manifold}(\Manifold, \R^{n+1}):= \{\ShapeEmbedding\in \operatorname{Emb}(\Manifold, \R^{n+1}): \ShapeEmbedding(p) = p \quad \forall p\in \partial\Manifold \}.
\end{equation}
For numerical routines this means that a specified boundary $\partial\Manifold$ of the starting shape is left fixed, whereas the interior of the shape is able to deform and change its shape and parameterization.
In the case of empty boundary, the pre-shape space becomes $\operatorname{Emb}(\Manifold, \R^{n+1})$, meaning shapes are allowed to move freely.

With the introduction of covariant derivatives and appropriate pre-shapes for the boundary case, we can formulate a \emph{pre-shape parameterization tracking problem} inspired by a least-squares formulation of the deformation method for mesh element volume optimization as found in \cite{cai2004adaptive, cao2002moving, grajewski2009mathematical}.
We remind the reader, that mesh deformation methods track for a specified target cell volume $f$ by changing coordinates of nodes.
For our formulation, we take a slight twist by using inverse Jacobians, which changes interpretation of optimal $\ShapeEmbedding$ and targets $f_\ShapeEmbedding$.
In our case, $f_\ShapeEmbedding$ describes the desired local density of mesh nodes, whereas the authors of \cite{cai2004adaptive, cao2002moving, grajewski2009mathematical} use targets $f$ describing the local cell volume.
For this reason the mentioned authors need to use reciprocals of $f$, instead of reciprocals of Jacobians.
Still, both formulations are equivalent by inverting the solutions, Jacobians and targets.
In addition to targets $f_\ShapeEmbedding$, we also incorporate a positive function $g^\Manifold: \Manifold \rightarrow (0,\infty)$, which will act as the distribution of nodes for the initial mesh.


The following proposition gives the definition, well-definedness and existence of solutions of the pre-shape parameterization tracking problem.
\begin{proposition}[Pre-Shape Parameterization Tracking Problem and Existence of Solutions]\label{Theorem_ExistenceGeneralProblem}
	Let $\Manifold$ be an $n$-dimensional, orientable,  path-connected and compact $C^\infty$-submanifold of $\R^{n+1}$, possibly with non-empty boundary $\partial\Manifold$ of $C^\infty$-regularity. 
	Additionally, let $g^\Manifold: \Manifold \rightarrow (0,\infty)$ and $f_\ShapeEmbedding: \ShapeEmbedding(\Manifold) \rightarrow (0,\infty)$ be $C^\infty$-functions, with $f$ having shape functionality. 
	Further assume the normalization condition
	\begin{equation}\label{Assumption_Theorem1fNormed}
	\int_{\ShapeEmbedding(\Manifold)} f_\ShapeEmbedding(s) \; \diff s = \int_\Manifold g^\Manifold(s)\; \diff s  \quad \forall \ShapeEmbedding\in \operatorname{Emb}_{\partial\Manifold}(\Manifold, \R^{n+1}).
	\end{equation}
	
	Then the following problem
	\begin{equation}\label{Eq_FinalGeneralParamTrackingProblem}
		\underset{\ShapeEmbedding \in \operatorname{Emb}_{\partial\Manifold}(\Manifold, \R^{n+1})}{\min} \frac{1}{2}\int_{\ShapeEmbedding(\Manifold)}
		\Big(
		g^\Manifold \circ\ShapeEmbedding^{-1}(s) \cdot \frac{1}{\det D^\tau \ShapeEmbedding \circ\ShapeEmbedding^{-1}(s)}
		-
		f_{\ShapeEmbedding}(s) 
		\Big)^2 \; \diff s
	\end{equation}
	is called \emph{pre-shape parameterization tracking problem}.
	It is well-defined and independent of choice of local orthornormal frames $\TangentVector$, $\TangentVector_\ShapeEmbedding$ on $\Manifold$ and $\ShapeEmbedding(\Manifold)$. 
	Further, in each fiber $\ProjectionCanonical(\ShapeEmbedding)$ there exists a global $C^\infty$-solution to problem  \cref{Eq_FinalGeneralParamTrackingProblem}, i.e. an embedding $\tilde{\ShapeEmbedding}$ satisfying
	\begin{equation}\label{Eq_EquiCharSol}
	 (g^\Manifold\circ\tilde{\ShapeEmbedding}^{-1})\cdot\det D^\tau \tilde{\ShapeEmbedding}^{-1} \equiv f_\ShapeEmbedding \quad \text{ and } \quad
	 \tilde{\ShapeEmbedding}(p) = p \quad \forall p\in \partial\tilde{\ShapeEmbedding}(\Manifold).
	\end{equation}
\end{proposition}

\begin{proof}
	The main ingredient of this proof is a theorem by Moser \cite{moser1965volume} from 1965, extended by Dacorogna and Moser in \cite[Thrm. 1]{dacorogna1990partial}, which guarantees existence solutions.
	Due to the quadratic nature of the objective functional it is obvious that \cref{Eq_EquiCharSol} is a sufficient condition for optimality.
	Together with normalization condition \cref{Assumption_Theorem1fNormed}, Moser's and Dacorogna's theorem guarantees existence of embeddings satisfying \cref{Eq_EquiCharSol}, which is a polynomial PDE by application of Laplace's formula for determinants.
	
	Fix an orientation for $\Manifold$, $\ShapeEmbedding \in \operatorname{Emb}_{\partial\Manifold}(\Manifold, \R^{n+1})$ and let $\TangentVector$, $\TangentVector^\ShapeEmbedding$ be local orthonormal frames for each $\ShapeEmbedding \in \operatorname{Emb}_{\partial\Manifold}(\Manifold, \R^{n+1})$.
	Well-definedness of the integrand in problem  \cref{Eq_FinalGeneralParamTrackingProblem} is clear, since $g^\Manifold$ is positive and $\ShapeEmbedding$ is an immersion, making $D^\tau \ShapeEmbedding\in GL(n,\R)$ and thus $\det D^\tau \ShapeEmbedding$ non-vanishing.
	In the case of non-empty boundary, the integral is well-defined by using the interior of $\ShapeEmbedding(\Manifold)$, as the boundary $\partial\Manifold$ is a set of measure zero.
	Independence of choice of orientation preserving local orthonormal bases inducing the covariant derivative (see \cref{Defi_CovariateDeriv}) is also clear, since an orientation preserving change of the orthonormal base can be realized by multiplications with orthogonal matrices $\tilde{B},B \in SO(n)$, and hence by the determinant product rule $\det D^\tau \ShapeEmbedding$ remains invariant.
	Further, if no global orthonormal frame exists,  well-definedness and independence of choice of local orthonormal frames is guaranteed as well.
	This can be ensured by using a partition of unity, which covers $\ShapeEmbedding(\Manifold)$ with open domains of local orthonormal frames, and linearity of the integral in \cref{Eq_FinalGeneralParamTrackingProblem} together with the previous argument about the change of orthonormal bases.
\end{proof}

\begin{remark}[H\"older Regularity Case]
	Existence and well-definedness results from \cref{Theorem_ExistenceGeneralProblem} also hold in the more general context of $C^{k,\alpha}$-H\"older regularity.
	For given $k\in\mathbb{N}$ and $\alpha\in(0,1)$, if $\Manifold$, $f_\ShapeEmbedding$ and $g^\Manifold$ have $C^{k,\alpha}$-regularity, and $\partial\Manifold$ has $C^{k+3, \alpha}$ regularity, then solutions $\ShapeEmbedding$ with $C^{k+1,\alpha}$-regularity exist in each fiber.
	This stems from the regularity results in \cite{dacorogna1990partial}.
\end{remark}

Having guaranteed existence of solutions, we want to turn our attention to the pre-shape derivative of the general parameterization tracking problem \cref{Eq_FinalGeneralParamTrackingProblem}.
This is serves several different purposes in our studies.
For one, it is of numerical interest, since we will construct several algorithms for improvement of mesh quality in shape optimization routines based on derivatives.
At the same time, \cref{Eq_FinalGeneralParamTrackingProblem} serves as a non-trivial example to illustrate the application of pre-shape calculus techniques developed in \cref{Section_GeneralTheory}.
In particular, we will see that the derivative to the general parameterization tracking problem \cref{Eq_FinalGeneralParamTrackingProblem} is not accessible via classical shape calculus techniques.
In the following, we will leave the target density functions $f_\ShapeEmbedding$ general and only assume enough regularity for existence of the pre-shape derivative.
Later in this article, we will propose an explicit way to construct $f_\ShapeEmbedding$ (cf. monitor function study \cite{cao1999study}), while also ensuring existence of its material derivatives with a closed
form (cf. \cref{Eq_MaterialDerivativefExtForce}).

As we permit non-empty boundaries left to be invariant in the parameterization tracking problem, the space of possible test functions $V$ is altered in an according way.
In particular, due to invariance of $\partial\Manifold$,  vector fields for $\operatorname{Emb}_{\partial\Manifold}(\Manifold, \R^{n+1})$ are given by vector fields vanishing on the boundary (cf. \cite[Thrm. 8.2]{ebin1970groups}, \cite[Thrm. 3.19]{smolentsev2007diffeomorphism}), i.e. 
\begin{equation}
	C^\infty_{\partial\Manifold}(\R^{n+1},\R^{n+1}) := \{V\in C^\infty(\R^{n+1},\R^{n+1}): \operatorname{Tr}_{\vert\partial\Manifold}(V)=0 \}.
\end{equation}
Of course, the same is true for H\"older and Sobolev regularities.
With this in mind, we can derive the pre-shape derivative of the parameterization tracking problem.

\begin{theorem}[Pre-Shape Derivative of the Pre-Shape Parameterization Tracking Problem]\label{Theorem_ShapeDeriv}
	Let the assumptions of \cref{Theorem_ExistenceGeneralProblem} hold and denote by $\TargetPrShp^\tau$ the objective functional of the general parameterization tracking problem \cref{Eq_FinalGeneralParamTrackingProblem}.
	Also, assume enough regularity for $f_\ShapeEmbedding$, such that material derivatives exist.
	
	Then, for fixed $\ShapeEmbedding\in \operatorname{Emb}_{\partial\Manifold}(\Manifold, \R^{n+1})$ and $V\in C^\infty_{\partial \ShapeEmbedding(\Manifold)}(\R^{n+1},\R^{n+1})$, the pre-shape derivative of \cref{Eq_FinalGeneralParamTrackingProblem} is given by 
	\begin{align}\label{Eq_ShapeDerivTang}
	\begin{split}
	\PrShpDeriv\TargetPrShp^\tau(\ShapeEmbedding)[V] 
	=&-\int_{\ShapeEmbedding(\Manifold)} \frac{1}{2}\cdot\Big(\big(g^\Manifold\circ\ShapeEmbedding^{-1}\cdot \frac{1}{\det D^\tau \ShapeEmbedding}\circ\ShapeEmbedding^{-1}\big)^2 - f_{\ShapeEmbedding}^2\Big)
	\cdot\operatorname{div}_\Shape (V) 
	\\
	&\qquad\qquad\quad+
	\Big(g^\Manifold\circ\ShapeEmbedding^{-1}\cdot \frac{1}{\det D^\tau \ShapeEmbedding}\circ\ShapeEmbedding^{-1} - f_\ShapeEmbedding\Big)
	\cdot
	\PrShpDeriv_m(f_\ShapeEmbedding)[V] \;\diff s,
	\end{split}
	\end{align}
	with $\PrShpDeriv_m(f_\ShapeEmbedding)$ being the pre-shape material derivative of $f_\ShapeEmbedding$ and  $\operatorname{div}_{\Gamma}$ the tangential divergence on $\ShapeEmbedding(\Manifold)$.
	The pre-shape derivative does not depend on the choice of oriented local orthonormal frames $\TangentVector, \TangentVector^\ShapeEmbedding$ for representing the covariant derivative $D^\tau$.
\end{theorem}

\begin{proof}
	For the proof we rely on pre-shape calculus rules we have established in \cref{Section_GeneralTheory}.
	In particular, we will make use of formulae found in  \cref{Cor_PreShapeCalc}. 
	So let $\Manifold$ fulfill the assumptions made in \cref{Theorem_ShapeDeriv}. 
	Fix a $\ShapeEmbedding\in\operatorname{Emb}_{\partial\Manifold}(\Manifold,\R^{n+1})$ and let $V\in C^{\infty}_{\partial \ShapeEmbedding(\Manifold)}(\R^{n+1},\R^{n+1})$. 
	The following arguments are all valid for $C^{k,\alpha}$-regularity.
	
	The use of pre-shape material derivative makes sense for families of differentiable functions on varying  domains $\{f_\ShapeEmbedding: \ShapeEmbedding(\Manifold) \rightarrow \R\}_{\ShapeEmbedding \in \operatorname{Emb}_{\partial\Manifold}(\Manifold, \R^{n+1})}$
	depending smoothly on $\ShapeEmbedding$, 
	since the limit $\PrShpDeriv_mf(\ShapeEmbedding)[V]$ involves the term $f_{\ShapeEmbedding_t}(x_t)$.
	An easy check reveals that the term is well defined due to \cref{Defi_PreShapeMaterialDerivDefi} of moving points $x_t$ and the perturbation of identity for pre-shapes \cref{Defi_PertuOfIdPreShapes} coinciding
	\begin{equation}
	x_t 
	=
	x_0 + t\cdot V(x_0) 
	=
	\ShapeEmbedding(\ShapeEmbedding^{-1}(x_0)) + t\cdot V\circ\ShapeEmbedding(\ShapeEmbedding^{-1}(x_0))
	=
	\ShapeEmbedding_t(\ShapeEmbedding^{-1}(x_0)) \in \ShapeEmbedding_t(\Manifold).
	\end{equation}
	However, note that in this case there is no decomposition of type \cref{Eq_MaterialDerivDecompo} for $\PrShpDeriv_m f(\ShapeEmbedding)$ , since $x_t\notin \ShapeEmbedding(\Manifold)$ and $x_0 \notin \ShapeEmbedding_t(\Manifold)$ in general.
	
	With this in mind, we can apply \cref{Cor_PreShapeCalc} (v) to \cref{Eq_FinalGeneralParamTrackingProblem} in order to get
	\begin{align}\label{Temp_PreShapeDeriv}
	\begin{split}
	\PrShpDeriv\TargetPrShp^\tau(\ShapeEmbedding)[V] = 
	\int_{\ShapeEmbedding(\Manifold)} \PrShpDeriv_m&\Big(\frac{1}{2}\big(g^\Manifold \circ\ShapeEmbedding^{-1}\cdot \frac{1}{\det D^\tau \ShapeEmbedding}\circ\ShapeEmbedding^{-1} - f_\ShapeEmbedding\big)^2\Big)[V] \\
	+ \frac{1}{2}&\operatorname{div}_{\Gamma}(V)\Big(g^\Manifold \circ\ShapeEmbedding^{-1}\cdot \frac{1}{\det D^\tau \ShapeEmbedding}\circ\ShapeEmbedding^{-1} - f_\ShapeEmbedding\Big)^2\;\diff s.
	\end{split}
	\end{align}
	For simplification of the material derivative of the integrand, we employ our assumption on existence of material derivatives for $f_\ShapeEmbedding$, together  with the chain- and product rule for material derivatives (cf. \cref{Cor_PreShapeCalc}), to see
	\begin{align}\label{Temp5}
	\begin{split}
	\PrShpDeriv_m&\Big(\frac{1}{2}\big(g^\Manifold \circ\ShapeEmbedding^{-1}\cdot \frac{1}{\det D^\tau \ShapeEmbedding}\circ\ShapeEmbedding^{-1} - f_\ShapeEmbedding\big)^2\Big)[V] \\
	=\; &(g^\Manifold \circ\ShapeEmbedding^{-1}\cdot \frac{1}{\det D^\tau \ShapeEmbedding}\circ\ShapeEmbedding^{-1} - f_\ShapeEmbedding\big)
	\cdot
	\Bigg(
	\PrShpDeriv_m(g^\Manifold \circ\ShapeEmbedding^{-1})[V]\cdot \frac{1}{\det D^\tau \ShapeEmbedding}\circ\ShapeEmbedding^{-1} \\
	\;\;&-\; g^\Manifold \circ\ShapeEmbedding^{-1} \cdot \frac{1}{(\det D^\tau \ShapeEmbedding)^2}\circ\ShapeEmbedding^{-1} \cdot \PrShpDeriv_m(\det D^\tau \ShapeEmbedding \circ \ShapeEmbedding^{-1})[V]
	- D_m(f_\ShapeEmbedding)[V]
	\Bigg).
	\end{split}
	\end{align}
	In the following we examine the remaining material derivatives in \cref{Temp5}, except for $\PrShpDeriv_m(f_\ShapeEmbedding)[V]$, as we let $f_\ShapeEmbedding$ remain general. 
	To avoid confusion, we remind the reader that we are confronted with mappings $h$ taking two arguments, one explicitly being a pre-shape, making them operators of the form
	\begin{align}
	h_\cdot(\cdot) := h(\cdot,\cdot):\; \operatorname{Emb}_{\partial\Manifold}(\Manifold,\R^{n+1})\times\Manifold \rightarrow \R,\; (\ShapeEmbedding,p) \mapsto h_\ShapeEmbedding(p). 
	\end{align}
	We will use the following relationship of embeddings and the domain perturbation of identity
	\begin{align}\label{Temp6}
	\ShapeEmbedding_t(p) = (T_t \circ \ShapeEmbedding)(p) \quad\forall p\in\Manifold
	\Leftrightarrow (\ShapeEmbedding_t^{-1}\circ T_t)(q) = \ShapeEmbedding^{-1}(q) \quad\forall q\in\ShapeEmbedding(\Manifold),
	\end{align}
	where $\ShapeEmbedding_t$ is the perturbation of identity for pre-shapes (cf. \cref{Defi_PertuOfIdPreShapes}) for sufficiently small $t>0$. 
	If material derivatives of $h$ are assumed to exist, \cref{Temp6} leads to the following elementary but interesting identity
	\begin{spacing}{1.}
	\begin{align}\label{Temp7}
	\PrShpDeriv_m\big(h_\ShapeEmbedding\circ\ShapeEmbedding^{-1}\big)[V] 
	=
	\frac{\diff}{\diff t}_{\vert t=0} h(\ShapeEmbedding_t, \ShapeEmbedding_t^{-1}\circ T_t)
	=
	\frac{\diff}{\diff t}_{\vert t=0} h(\ShapeEmbedding_t, \ShapeEmbedding^{-1}) 
	= 
	\PrShpDeriv(h_\ShapeEmbedding)[V]\circ\ShapeEmbedding^{-1}.
	\end{align} 
	\end{spacing}	
	Applying this to the first remaining material derivative in \cref{Temp5}, we get
	\begin{spacing}{1.}
	\begin{align}\label{Temp_MaterialDerivg}
	\begin{split}
	\PrShpDeriv_m\big(g^\Manifold\circ\ShapeEmbedding^{-1}\big)[V] 
	= 
	\PrShpDeriv(g^\Manifold)[V]\circ\ShapeEmbedding^{-1} 
	= 0,
	\end{split}
	\end{align}
	\end{spacing}
	since $g^\Manifold$ is does not depend on choice of $\ShapeEmbedding\in\operatorname{Emb}(\Manifold, \R^{n+1})$.
	
	Next, we apply analogous techniques to the second material derivative. 
	Hence, for calculation of the material derivative of $\det D^\tau\ShapeEmbedding \circ \ShapeEmbedding^{-1}$ it is sufficient to calculate its pre-shape derivative. 
	Also, since the flow $(\ShapeEmbedding_t)_{t\in[0,\varepsilon)}$ given by the perturbation of identity in direction $V$ (cf. \cref{Defi_PertuOfIdPreShapes}) is differentiable in $t$, we can employ Jacobi's formula for the derivative of the determinant at $t_0=0$ to arrive at
	\begin{spacing}{1.}
	\begin{align}\label{Temp8}
	\begin{split}
	\PrShpDeriv_m\big(\det D^\tau\ShapeEmbedding \circ \ShapeEmbedding^{-1}\big)[V]
	&= \Big(\frac{\diff}{\diff t}_{\vert t=t_0}\det D^\tau\ShapeEmbedding_t\Big)\circ \ShapeEmbedding^{-1}\\
	&=\operatorname{tr}\Big(\operatorname{Adju}(D^\tau\ShapeEmbedding_{t_0}) \frac{\diff}{\diff t}_{\vert t=t_0} D^\tau\ShapeEmbedding_t\Big)\circ\ShapeEmbedding^{-1} \\
	&\hspace{-0.15cm}\overset{t_0=0}{=}\operatorname{tr}\Big(\operatorname{Adju}(D^\tau \ShapeEmbedding) D^\tau(V\circ\ShapeEmbedding) \Big)\circ\ShapeEmbedding^{-1},
	\end{split}
	\end{align}
	\end{spacing}
	where $\operatorname{Adju}(\cdot)$ is the adjugate matrix and $\operatorname{tr}(\cdot)$ is the trace operator for matrices.

	Knowing $D^\tau \ShapeEmbedding$ is invertible for all $p \in \Manifold$ due to $\ShapeEmbedding \in \operatorname{Emb}(\Manifold, \R^{n+1})$, we can use Cramer's rule to express the adjugate in terms of inverses.
	Also, we can use invariance of the trace operator under permutations of multiplicative order of matrices, giving us
	\begin{spacing}{1.}
	\begin{align}\label{Temp8_new}
	\begin{split}
		\operatorname{tr}\Big(\operatorname{Adju}(D^\tau \ShapeEmbedding) D^\tau(V\circ\ShapeEmbedding) \Big)\circ\ShapeEmbedding^{-1} 
		&= 
		\operatorname{tr}\Big(\operatorname{det} (D^\tau\ShapeEmbedding)
		\cdot
		D^\tau \ShapeEmbedding^{-1} D^\tau(V\circ\ShapeEmbedding) \Big)\circ\ShapeEmbedding^{-1} \\
		&=
		(\operatorname{det} D^\tau\ShapeEmbedding) \circ\ShapeEmbedding^{-1} \cdot
		\operatorname{tr}
		\Big(
		(D^\tau \ShapeEmbedding)^{-1} D^\tau V(\ShapeEmbedding) D^\tau\ShapeEmbedding \Big)\circ\ShapeEmbedding^{-1} \\
		&=
		(\operatorname{det} D^\tau\ShapeEmbedding) \circ\ShapeEmbedding^{-1} \cdot
		\operatorname{tr}
		\Big(D^\tau V(\ShapeEmbedding) \Big)\circ\ShapeEmbedding^{-1}	\\
		&=
		(\operatorname{det} D^\tau\ShapeEmbedding) \circ\ShapeEmbedding^{-1} 
		\cdot
		\operatorname{div}_\Shape(V).
	\end{split}
	\end{align}	
	\end{spacing}
	Using \cref{Temp_MaterialDerivg} and \cref{Temp8_new} in \cref{Temp5}, and plugging the material derivative into \cref{Temp_PreShapeDeriv}, we arrive at 
	\begin{align}
	\begin{split}
	\PrShpDeriv\TargetPrShp^\tau(\ShapeEmbedding)[V] =& 
	\int_{\ShapeEmbedding(\Manifold)} 
	\big(g^\Manifold \circ\ShapeEmbedding^{-1}\cdot \frac{1}{\det D^\tau \ShapeEmbedding}\circ\ShapeEmbedding^{-1} - f_\ShapeEmbedding\big) \\
	&\qquad\cdot
	\Big( -g^\Manifold\circ\ShapeEmbedding^{-1}\cdot\frac{1}{\det D^\tau \ShapeEmbedding}\circ\ShapeEmbedding^{-1}\cdot \operatorname{div}_\Shape(V) - \PrShpDeriv_m(f_\ShapeEmbedding)[V]
	\Big)
	\\
	&\qquad + \frac{1}{2}\operatorname{div}_{\Shape}(V)\Big(g^\Manifold \circ\ShapeEmbedding^{-1}\cdot \frac{1}{\det D^\tau \ShapeEmbedding}\circ\ShapeEmbedding^{-1} - f_\ShapeEmbedding\Big)^2\;\diff s 
	\\
	=&-\int_{\ShapeEmbedding(\Manifold)} 
	\Big(g^\Manifold\circ\ShapeEmbedding^{-1}\cdot\frac{1}{\det D^\tau \ShapeEmbedding}\circ\ShapeEmbedding^{-1} - f_\ShapeEmbedding\Big) \\
	&\qquad \cdot \Bigg(
	\frac{1}{2}\Big(g^\Manifold\circ\ShapeEmbedding^{-1}\cdot\frac{1}{\det D^\tau \ShapeEmbedding}\circ\ShapeEmbedding^{-1} + f_\ShapeEmbedding\Big)\cdot\operatorname{div}_\Shape(V) 
	+  \PrShpDeriv_m(f_\ShapeEmbedding)[V]
	\Bigg)\;\diff s \\
	=&-\int_{\ShapeEmbedding(\Manifold)} \frac{1}{2}\cdot\Big((g^\Manifold\circ\ShapeEmbedding^{-1}\cdot \frac{1}{\det D^\tau \ShapeEmbedding}\circ\ShapeEmbedding^{-1})^2 - f_{\ShapeEmbedding}^2\Big)
	\cdot\operatorname{div}_\Shape (V) 
	\\
	&\qquad\qquad+
	\Big(g^\Manifold\circ\ShapeEmbedding^{-1}\cdot \frac{1}{\det D^\tau \ShapeEmbedding}\circ\ShapeEmbedding^{-1} - f_\ShapeEmbedding\Big)
	\cdot
	\PrShpDeriv_m(f_\ShapeEmbedding)[V] \;\diff s,
	\end{split}
	\end{align}
	which is the desired pre-shape derivative \cref{Eq_ShapeDerivTang}.
	The covariant derivative $D^\tau \ShapeEmbedding$, and hence also the pre-shape derivative \cref{Eq_ShapeDerivTang}, is  independent of choice of orthonormal frames by similar reasoning as in the first part of the proof to  \cref{Theorem_ExistenceGeneralProblem}.
\end{proof}

In general cases, $\PrShpDeriv\TargetPrShp^{\tau}(\ShapeEmbedding)[V]$ is non-vanishing for vector fields $V$ tangential to $\ShapeEmbedding(\Manifold)$. 
By structure theorem for pre-shape derivatives \cref{Theorem_PreShapeHadamard}, globally vanishing tangential pre-shape derivatives indicate a functional which is almost of classical shape functional type. 
If we take the form of the pre-shape derivative \cref{Eq_ShapeDerivTang}, this clearly means \cref{Eq_FinalGeneralParamTrackingProblem} cannot be formulated as a shape optimization problem, nor is it tractable by classical shape calculus.

In light of the main structure \cref{Theorem_PreShapeHadamard} for pre-shape derivatives, we can further refine the representation of the pre-shape derivative \cref{Eq_ShapeDerivTang} by decomposing it into normal and tangential parts.
Interestingly, if the user wants to optimize for mesh quality by using pre-shape derivative based parameterization tracking, it is not recommendable to use the full pre-shape derivative found in equation  \cref{Eq_ShapeDerivTang}.
Instead, by decomposing the pre-shape derivative, we will see that the tangential component is sufficient for this task.
If a special case of the pre-shape derivatives normal component is used, we recover numerical methods solving Plateau's problem by constructing minimal surfaces (cf. \cite{dziuk1990algorithm, pinkall1993computing,  dziuk1999discrete}).
In some sense orthogonal to this, the use of tangential components gives algorithms resembling the deformation method for opimization of mesh quality (cf. \cite{liao1992new, grajewski2009mathematical, cao2002moving}).

\paragraph{Decomposing the Pre-Shape Derivative}

To derive this decomposition, notice the following informal relationship between tangential divergence and the mean curvature $\kappa$ for hypersurfaces for $V\in C^\infty(\R^{n+1}, \R^{n+1})$ (cf.\cite[Defi. 4.23]{Lee})
\begin{spacing}{1.}
\begin{equation}
\operatorname{div}_\Shape(\langle V, n \rangle \cdot n) 
=
(\nabla_\Shape \langle V, n \rangle )^T n + \langle V, n \rangle\cdot \operatorname{div}_\Shape(n)
= 
\operatorname{dim}(\Manifold) \cdot\langle V, n \rangle \cdot \kappa, 
\end{equation}	
\end{spacing}
due to orthogonality of tangential gradients $\nabla_\Shape(\langle V, n \rangle)$ and the outer normal vector field $n$ on the interior of $\ShapeEmbedding(\Manifold)$.
Also, let us briefly assume $f_\ShapeEmbedding: \R^{n+1} \rightarrow \R$ mapping from the whole ambient space, which simplifies using normal derivatives of $f_\ShapeEmbedding$ for the decomposition.
With this, and the assumption of constant target parameterizations for each fiber, i.e.  $f_{\ShapeEmbedding} = f_{\ShapeEmbedding\circ\Diffeomorphism}$ for all $\Diffeomorphism\in\operatorname{Diff}(\Manifold)$, we can refine  \cref{Eq_ShapeDerivTang} to
\begin{equation}\label{Eq_PrShpDerivativeTrackingDecompo}
\PrShpDeriv\TargetPrShp^\tau(\ShapeEmbedding)[V] = \langle g_\ShapeEmbedding^\NormalSpaceShape, V \rangle + \langle g_\ShapeEmbedding^\TangentSpaceShape, V \rangle \quad \forall V\in C^\infty(\R^{n+1}, \R^{n+1}),
\end{equation}
with shape (i.e. normal) component
\begin{align}\label{Eq_PreShapeDerTrackingNormal}
\begin{split}
\langle g_\ShapeEmbedding^\NormalSpaceShape, V \rangle 
= 
&-\int_{\ShapeEmbedding(\Manifold)} \frac{\operatorname{dim}(\Manifold)}{2}\cdot\Big(\big(g^\Manifold\circ\ShapeEmbedding^{-1}\cdot \det D^\tau \ShapeEmbedding^{-1}\big)^2 - f_{\ShapeEmbedding}^2\Big)
\cdot\kappa\cdot \langle V, n\rangle 
\\
&\qquad\qquad\quad+
\Big(g^\Manifold\circ\ShapeEmbedding^{-1}\cdot \det D^\tau \ShapeEmbedding^{-1} - f_\ShapeEmbedding\Big)
\cdot
\Big(
\frac{\partial f_\ShapeEmbedding}{\partial n} \cdot \langle V, n \rangle + \ShpDeriv(f_\ShapeEmbedding)[V]
\Big) \;\diff s 
\end{split}
\end{align}
and pre-shape (i.e. tangential) component
\begin{align}\label{Eq_PreShapeDerTrackingTangential}
\begin{split}
\langle g_\ShapeEmbedding^\TangentSpaceShape, V \rangle 
= 
&-\int_{\ShapeEmbedding(\Manifold)} \frac{1}{2}\cdot\Big(\big(g^\Manifold\circ\ShapeEmbedding^{-1}\cdot \det D^\tau \ShapeEmbedding^{-1}\big)^2 - f_{\ShapeEmbedding}^2\Big)
\cdot\operatorname{div}_\Shape (V - \langle V, n \rangle \cdot n) 
\\
&\qquad\qquad\quad+
\Big(g^\Manifold\circ\ShapeEmbedding^{-1}\cdot \det D^\tau \ShapeEmbedding^{-1} - f_\ShapeEmbedding\Big)
\cdot
\nabla_\Shape f_\ShapeEmbedding^T V \;\diff s.
\end{split}
\end{align}
Here, $\ShpDeriv(f_\ShapeEmbedding)$ is the classical shape derivative of $f_\ShapeEmbedding$.
The first integral corresponds to the classical shape derivative component $g^{\NormalSpaceShape}$ of decomposition \cref{Thrm_HadamardSplittingPreShDir} acting on normal directions. 
Next, the second integral acts on tangential directions and therefore corresponds to the parameterization part $g^{\TangentSpaceShape}$ in \cref{Thrm_HadamardSplittingPreShDir}.

\paragraph{Normal Component: Minimal Surfaces}
For illustration, let us deviate from normalization of the target (cf. \cref{Assumption_Theorem1fNormed}) by chosing $f_\ShapeEmbedding=0$ and $g^\Manifold=1$ for all $\ShapeEmbedding\in\operatorname{Emb}_{\partial\Manifold}(\Manifold, \R^{n+1})$.

In this situation the classical shape derivative component of $\PrShpDeriv\TargetPrShp^\tau(\ShapeEmbedding)$ is given by
\begin{equation}\label{Eq_PreShapeDerivHorizontalMinimalSurf}
\langle g_{\ShapeEmbedding}^{\NormalSpaceShape},V\rangle = -\frac{\operatorname{dim}(\Manifold)}{2}\cdot\int_{\ShapeEmbedding(\Manifold)}\big(\operatorname{det}D^\tau \ShapeEmbedding^{-1}(s)\big)^2\cdot \kappa(s)\cdot \big\langle V(s), n(s)\big\rangle \; \diff s.
\end{equation}
Since $\ShapeEmbedding$ are embeddings and $\Manifold$ compact, according Jacobians are bounded and non-vanishing.
In our special situation this means the horizontal component of the pre-shape derivative \cref{Eq_PreShapeDerivHorizontalMinimalSurf} is vanishing exactly for shapes with vanishing mean curvature $\kappa$.
Put differently, minimal surfaces and their higher dimension analogues are exactly the stationary points for this horizontal component.

Hence to this observation, a gradient ascend using \cref{Eq_PreShapeDerivHorizontalMinimalSurf} resembles an algorithm for evolutionary surfaces proposed by Dziuk in \cite{dziuk1990algorithm}, which solves Plateau's problem by approximating a mean curvature flow.
Note that an ascend is necessary, since our formulation of the pre-shape parameterization tracking problem involves inverse Jacobians,
which is connected to Plateau's problem by
	\begin{spacing}{1.0}
	\begin{equation}
	\max_{\ShapeEmbedding \in \operatorname{Emb}_{\partial\Manifold}(\Manifold, \R^{n+1})} \int_{\ShapeEmbedding(\Manifold)} \big(\operatorname{det}D^\tau \ShapeEmbedding^{-1}(s)\big)^2  \; \diff s 
	\Leftrightarrow
	\min_{\ShapeEmbedding \in \operatorname{Emb}_{\partial\Manifold}(\Manifold, \R^{n+1})} 
	\int_{\Manifold} \vert \operatorname{det}D^\tau \ShapeEmbedding(s) \vert \; \diff s.
	\end{equation}
	\end{spacing}
This also illustrates qualitative properties of a steepest descent using the complete pre-shape derivative.
Briefly stated, if less  vertices are desired at a location, the gradient descent in normal direction tends to blow up the shape, increasing distances of neighboring vertices.
If more vertices are desired it tends to locally flatten the shape, driving the nodes together.
This shows that application of the full pre-shape derivative distorts the shape in normal direction, hence interfering with the actual shape optimization problem to be regularized.

Use of the tangential component in \cref{Eq_PrShpDerivativeTrackingDecompo} leads to algorithms similar to the mesh deformation methods from \cite{liao1992new, grajewski2009mathematical, cao2002moving}.
Such a routine is discussed and implemented in the numerical \cref{Section_NumericsGeneral}.
We will use it in upcoming works to construct various regularization methods for shape optimization problems.

An illustrative numerical example of a pre-shape derivative for the parameterization tracking problem and its decomposition are shown in \cref{Fig_SurfGradDecompo} for target $f_\ShapeEmbedding(x, y, z) \equiv \frac{1}{\int_{\ShapeEmbedding(\Manifold)}1\;\diff s}\cdot x$ and a sphere centered at $(0.5,0.5,0.5)$.
\begin{figure}[h]
\begin{tabular}{cr}
		\begin{subfigure}{0.5\textwidth}
		\includegraphics[width=0.9\linewidth, height=5cm]{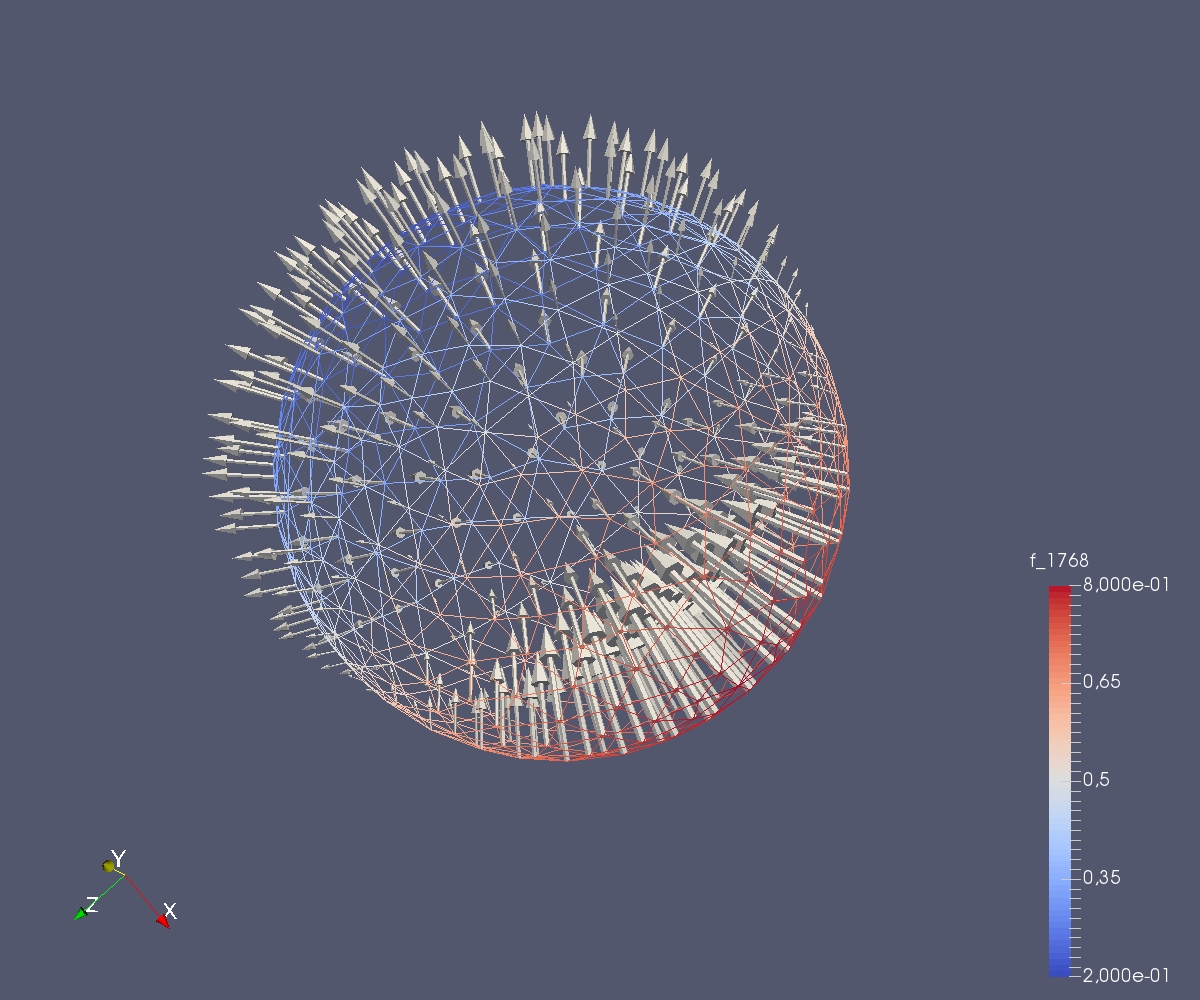}
		\subcaption{Normal component representing the classical shape part}
		\end{subfigure} 

		&\begin{subfigure}{0.5\textwidth}
		\includegraphics[width=0.9\linewidth, height=5cm]{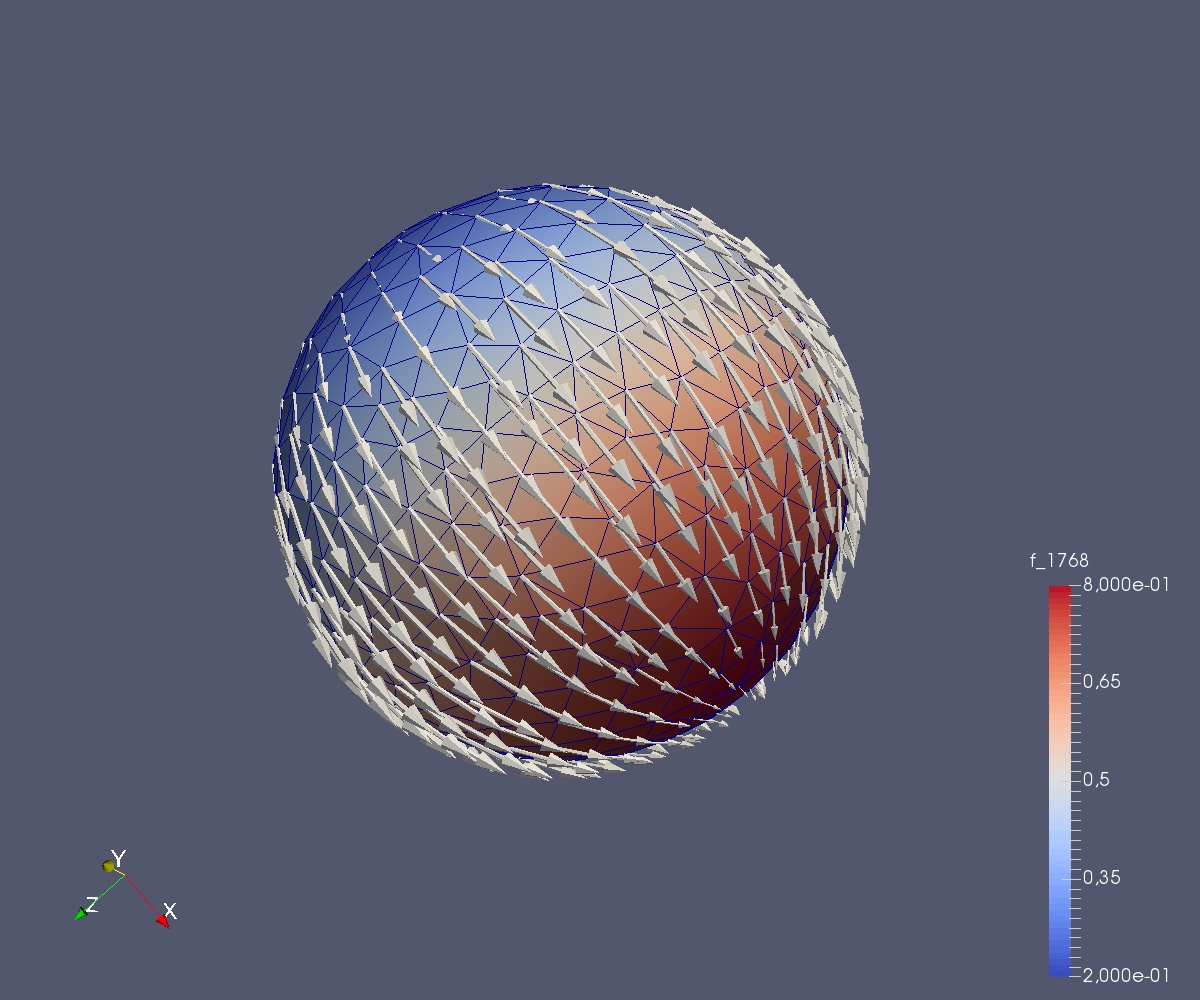}
		\subcaption{Tangential component representing the parameterization part}
		\end{subfigure} \\

		\begin{subfigure}{0.5\textwidth}
		\includegraphics[width=0.9\linewidth, height=5cm]{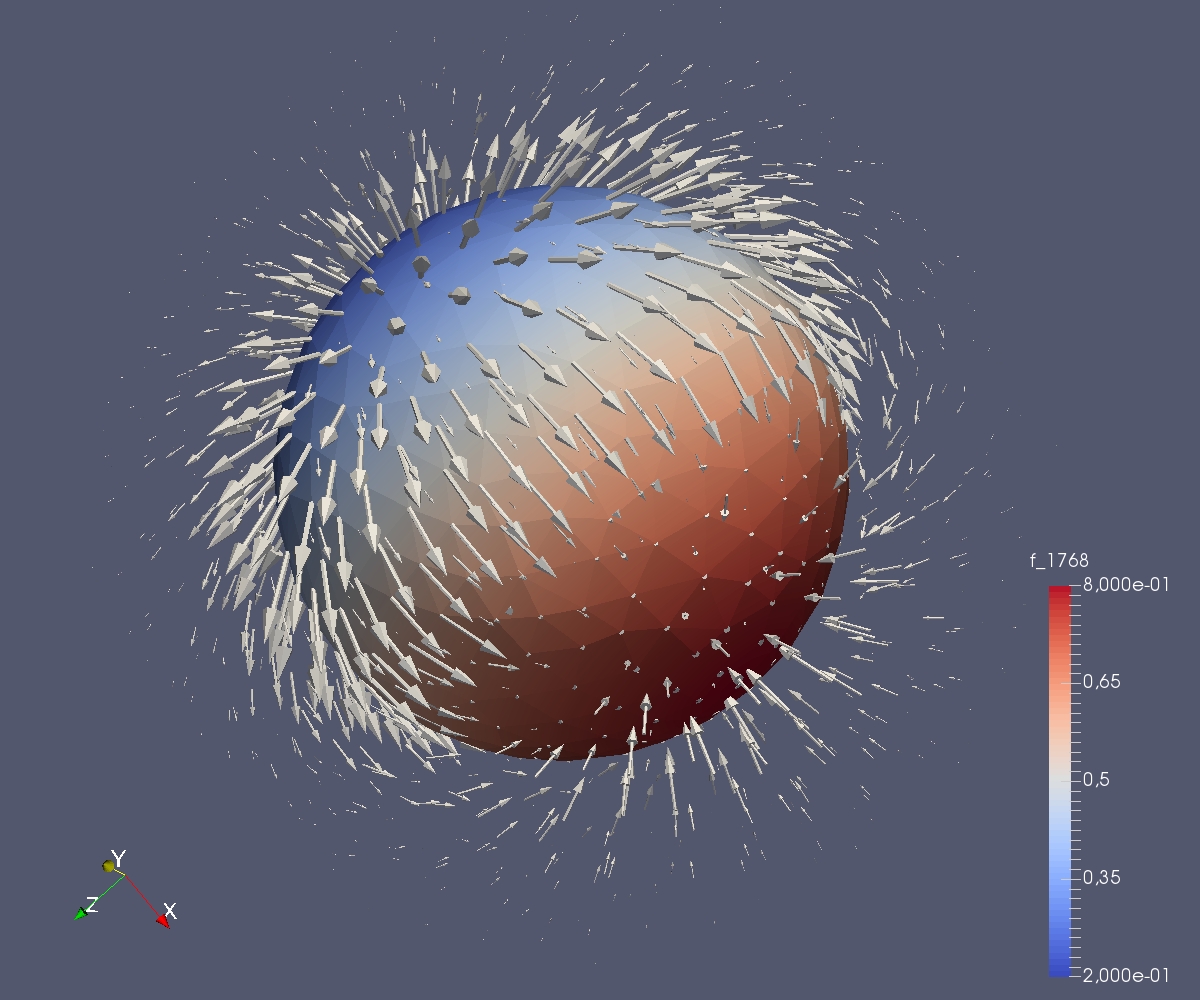}
		\subcaption{Complete pre-shape gradient in volume mesh representation}
		\end{subfigure}

		&\begin{subfigure}{0.5\textwidth}
		\includegraphics[width=0.9\linewidth, height=5cm]{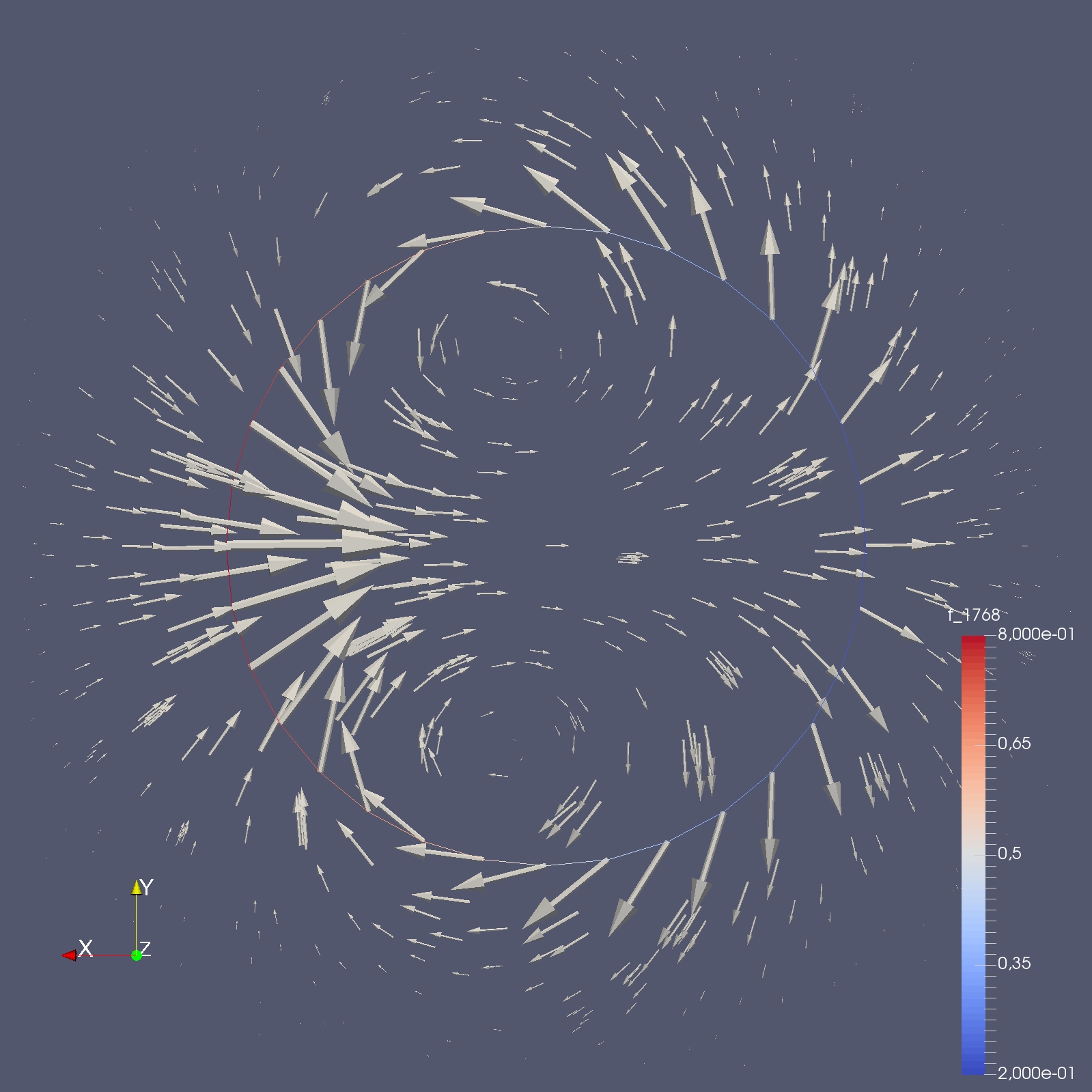}
		\subcaption{Slice through $xy$-plane at center of complete volume pre-shape gradient}
		\end{subfigure}
\end{tabular}
	\caption{\label{Fig_SurfGradDecompo}Negative pre-shape gradient of $\TargetPrShp^\tau$ on a sphere scaled by $0.02$ using target $f_\ShapeEmbedding(x, y, z) \equiv \frac{1}{\int_{\ShapeEmbedding(\Manifold)}1\;\diff s}\cdot x$, which is depicted by color, and gradient representation by a linear elasticity metric. Color shifting towards red means higher desire fore more volume/vertex allocation.}	
\end{figure}

\paragraph{A class of externally defined targets $f_\ShapeEmbedding$}
As we have left the target $f_\ShapeEmbedding$ unspecified during derivation of the pre-shape derivative \cref{Eq_ShapeDerivTang}, we want to give an constructive example, which can be implemented in numerical routines.
To do so, we have to keep in mind that the normalization requirement \cref{Assumption_Theorem1fNormed} on $f_\ShapeEmbedding$ has to be fulfilled.
One way to accomplish this, is by defining $f_\ShapeEmbedding$ using a given globally defined function $\RieszEnergyExtForce: \R^{n+1} \rightarrow (0,\infty)$.
By assuming $H^2$-regularity for $\RieszEnergyExtForce$, existence of pre-shape derivatives and their closed form as in \cref{Eq_ShapeDerivTang} is guaranteed.
The according target vertex point density on a shape $\ShapeEmbedding(\Manifold)$ is then given by
\begin{spacing}{1.0}
\begin{equation}\label{Eq_ExternalForceNormed}
f_{\ShapeEmbedding} 
=
\frac{\int_{\Manifold}g^{\Manifold}\;\diff s}{\int_{\ShapeEmbedding(\Manifold)}\RieszEnergyExtForce_{\vert\ShapeEmbedding(\Manifold)}\;\diff s}\cdot \RieszEnergyExtForce_{\vert\ShapeEmbedding(\Manifold)},
\end{equation}
\end{spacing}
which is well-defined due to the trace theorem for Sobolev functions. 
	
If the target $f_\ShapeEmbedding$ is chosen such that normalization \cref{Assumption_Theorem1fNormed} is not fulfilled, solutions to the parameterization tracking problem \cref{Eq_FinalGeneralParamTrackingProblem} might still exist. 
Depending on whether $\int_{\Manifold}g^\Manifold\;\diff s$ is greater or smaller $\int_\Manifold f_\ShapeEmbedding\;\diff s$, the gradient flow generated by the pre-shape derivative \cref{Eq_ShapeDerivTang} locally shrinks or blows up the shape $\ShapeEmbedding(\Manifold)$ in normal direction to compensate for the difference.

Next we calculate $\PrShpDeriv_m(f_{\ShapeEmbedding})[V]$, which exists since we have $\RieszEnergyExtForce_{\vert\ShapeEmbedding(\Manifold)}\in H^1\big(\ShapeEmbedding(\Manifold)\big)$.
For this, we apply pre-shape calculus rules established in \cref{Cor_PreShapeCalc}.
Also, since the external force $\RieszEnergyExtForce: \R^{n+1}\rightarrow (0,\infty)$ is defined on the entire ambient space, we can make direct use of decomposition for pre-shape material derivatives \cref{Cor_MaterialDeriv}.
Together with the fact that $\RieszEnergyExtForce$ and $g^\Manifold$ do not depend on $\ShapeEmbedding$, this gives
\begin{align}\label{Eq_MaterialDerivativefExtForce}
\begin{split}
\PrShpDeriv_m(f_{\ShapeEmbedding})[V] 
&=
\PrShpDeriv_m\Bigg(		\frac{\int_{\Manifold}g^{\Manifold}\;\diff s}{\int_{\ShapeEmbedding(\Manifold)}\RieszEnergyExtForce\;\diff s}\cdot \RieszEnergyExtForce\Bigg)[V] \\
&= -\frac{\int_{\Manifold}g^{\Manifold}\;\diff s}{\big(\int_{\ShapeEmbedding(\Manifold)}\RieszEnergyExtForce\;\diff s\big)^2}
\cdot\RieszEnergyExtForce\cdot
\int_{\ShapeEmbedding(\Manifold)}\big(\PrShpDeriv_m(\RieszEnergyExtForce)[V] 
+ \operatorname{div}_{\Shape}(V)\cdot\RieszEnergyExtForce\big)\;\diff s  \\
&\qquad+ 
\frac{\int_{\Manifold}g^{\Manifold}\;\diff s}{\int_{\ShapeEmbedding(\Manifold)}\RieszEnergyExtForce\;\diff s} 
\cdot 
\Big(\PrShpDeriv(\RieszEnergyExtForce)[V] + \nabla \RieszEnergyExtForce^T V\Big)	\\
&=  -\frac{\int_{\Manifold}g^{\Manifold}\;\diff s}{\big(\int_{\ShapeEmbedding(\Manifold)}\RieszEnergyExtForce\;\diff s\big)^2}
\cdot\RieszEnergyExtForce\cdot
\int_{\ShapeEmbedding(\Manifold)}\Big(
\nabla\RieszEnergyExtForce^T V
+
\operatorname{div}_{\Shape}(V)
\cdot
\RieszEnergyExtForce\Big)\;\diff s
+  
\frac{\int_{\Manifold}g^{\Manifold}\;\diff s}{\int_{\ShapeEmbedding(\Manifold)}\RieszEnergyExtForce\;\diff s} 
\cdot 
\nabla \RieszEnergyExtForce^T V \\
&=
-\frac{\int_{\Manifold}g^{\Manifold}\;\diff s}{\big(\int_{\ShapeEmbedding(\Manifold)}\RieszEnergyExtForce\;\diff s\big)^2}
\cdot\RieszEnergyExtForce\cdot
\int_{\ShapeEmbedding(\Manifold)} \frac{\partial \RieszEnergyExtForce}{\partial n}\cdot\langle V, n \rangle \;\diff s
+
\frac{\int_{\Manifold}g^{\Manifold}\;\diff s}{\int_{\ShapeEmbedding(\Manifold)}\RieszEnergyExtForce\;\diff s} 
\cdot
\nabla \RieszEnergyExtForce^T V,
\end{split}
\end{align}
where the last equality stems from Stokes theorem and our assumption on $V$ to vanish on the boundary $\partial \ShapeEmbedding(\Manifold)$.

	Next, we illustrate how the closed pre-shape derivative formula can be used to derive additional important properties of pre-shape optimization problems.
	In particular, we will see that local and global solutions parameterization tracking in each fiber coincide.	
	\begin{proposition}[Characterization of Global Solutions by Fiber  Stationarity]\label{Theorem_GlobalSolutionParamTrackingStationaryPoints}
		Let assumptions of \cref{Theorem_ShapeDeriv} be satisfied.
		Then the following statements are equivalent:
		\begin{enumerate}
			\item[(i)]  $\ShapeEmbedding\in \operatorname{Emb}(\Manifold, \R^{n+1})$ is a fiber stationary point of \cref{Eq_FinalGeneralParamTrackingProblem}, i.e.
			\begin{equation}\label{Eq_FiberCriticalPoint}
			\PrShpDeriv\TargetPrShp^\tau(\ShapeEmbedding)[V] = 0 \quad \forall V\in C_{\partial\Manifold}^\infty(\R^{n+1}, \R^{n+1}) \text{ with } \langle \operatorname{Tr}_{\vert \operatorname{int} \ShapeEmbedding(\Manifold)}[V], n\rangle_2 = 0,
			\end{equation}
			where $\operatorname{int} \ShapeEmbedding(\Manifold)$ is the interior of $\ShapeEmbedding(\Manifold)$ and $n$ is the outer normal field on  $\operatorname{int}\ShapeEmbedding(\Manifold)$.
			\item[(ii)] $\ShapeEmbedding$ is a global solution to \cref{Eq_FinalGeneralParamTrackingProblem}, and in particular it satisfies 
			\begin{equation}\label{Temp19}
			g^\Manifold\circ\ShapeEmbedding^{-1}\cdot \det D^\tau \ShapeEmbedding^{-1} = f_{\ShapeEmbedding} \; \text{ on } \ShapeEmbedding(\Manifold).
			\end{equation}
			\item[(iii)] the complete pre-shape derivative of $\TargetPrShp^\tau$ vanishes in $\ShapeEmbedding$, i.e.
			\begin{equation}\label{Eq_Stationarity}
				\PrShpDeriv\TargetPrShp^\tau(\ShapeEmbedding) = 0 \quad \forall V\in C_{\partial\Manifold}^\infty(\R^{n+1}, \R^{n+1}).
			\end{equation}
			Additionally, its normal component $g_{\ShapeEmbedding}^\NormalSpaceShape$ in $\ShapeEmbedding$ vanishes.
		\end{enumerate}
		In particular, the necessary first order condition regarding only directions $V$ tangential to $\ShapeEmbedding(\Manifold)$ is already a sufficient condition for being a global minimizer to $\TargetPrShp^\tau$.
	\end{proposition}
	
	\begin{proof}
		Let us assume the setting of \cref{Theorem_ShapeDeriv}. 
		We show equivalence of all assertions by a circular argument '$(i) \implies (ii) \implies (iii) \implies (i)'$.
		
		As a start, implication '$(ii) \implies (iii)$' is trivial.
		By assuming (ii), we can use relation \cref{Temp19} to see that the two integrands of $\PrShpDeriv\TargetPrShp^\tau(\ShapeEmbedding)[V]$ (cf. \cref{Eq_ShapeDerivTang}) featuring $g^\Manifold\circ\ShapeEmbedding^{-1}\cdot \det D^\tau \ShapeEmbedding^{-1}$ and $f_{\ShapeEmbedding}$ are zero for all directions $V\in C_{\partial\Manifold}^\infty(\R^{n+1}, \R^{n+1})$.
		The same argument applies for the normal component $g_\ShapeEmbedding^\NormalSpaceShape$ by using the explicit decomposition \cref{Eq_PrShpDerivativeTrackingDecompo}.
		Hence we immediately get \cref{Eq_Stationarity}.
		
		The non-trivial part is to prove '$(i) \implies (ii)$'.
		Let us assume (i) by fixing a $\ShapeEmbedding\in \operatorname{Emb}(\Manifold, \R^{n+1})$ satisfying fiber stationarity \cref{Eq_FiberCriticalPoint}.
		With the pre-shape derivative formula from  \cref{Theorem_ShapeDeriv} at hand, we can apply an integration by parts on manifolds (cf. \cite[Ch. 2.2, Proposition 2.3]{taylor2013partial1}), either using that $\Manifold$ is closed or $V$ is vanishing on the boundary, to get
		\begin{align}\label{Temp14}
		\begin{split}
		\PrShpDeriv\TargetPrShp^\tau(\ShapeEmbedding)[V] =&
		\int_{\ShapeEmbedding(\Manifold)} \Bigg((g^\Manifold\circ\ShapeEmbedding^{-1}\cdot \det D^\tau\ShapeEmbedding^{-1})\cdot\Big(\nabla_\Shape(g^\Manifold\circ\ShapeEmbedding^{-1}\cdot \det D^\tau\ShapeEmbedding^{-1}) - \nabla_\Shape f_{\ShapeEmbedding} \Big)\Bigg)^T
		V \;\diff s,
		\end{split}
		\end{align}
		where we also used the assumption in \cref{Theorem_ExistenceGeneralProblem} that $f_\ShapeEmbedding$ is constant in each fiber and $\langle \operatorname{Tr}_{\vert \ShapeEmbedding(\Manifold)}[V], n\rangle_2 = 0$ to reformulate $\PrShpDeriv_m(f_\ShapeEmbedding)[V]$.
		Due to assumption \cref{Eq_FiberCriticalPoint} of fiber stationarity, we know \cref{Temp14} equals zero for all $V$ tangential on $\ShapeEmbedding(\Manifold)$ up to the boundary.
		So on the interior of $\ShapeEmbedding(\Manifold)$ we get
		\begin{equation}
		(g^\Manifold\circ\ShapeEmbedding^{-1}\cdot \det D^\tau\ShapeEmbedding^{-1} )
		\cdot \Big(
		\nabla_\Shape\big(g^\Manifold\circ\ShapeEmbedding^{-1}\cdot \det D^\tau\ShapeEmbedding^{-1}\big) - \nabla_\Shape f_{\ShapeEmbedding}
		\Big)
		\equiv 0.
		\end{equation}
		By assumption we have $g^\Manifold >0$, so together with non-vanishing determinant by  $\ShapeEmbedding\in\operatorname{Emb}(\Manifold, \R^{n+1})$, this implies
		\begin{equation}\label{Temp15}
		\nabla_\Shape(g^\Manifold\circ\ShapeEmbedding^{-1}\cdot \det D^\tau\ShapeEmbedding^{-1} - f_{\ShapeEmbedding}) \equiv 0.
		\end{equation}
		Since involved functions are Lipschitz continuous, and 
		as \cref{Theorem_GlobalSolutionParamTrackingStationaryPoints} assumes $\Manifold$ to be smooth and path connected, we can derive constancy of the involved term.
		However, after using a pull-back and normalization assumption \cref{Assumption_Theorem1fNormed}, we see that the discussed constant is $0$, giving us 
		\begin{equation}
		g^\Manifold\circ\ShapeEmbedding^{-1}\cdot \det D^\tau (\ShapeEmbedding)^{-1} 
		=
		f_{\ShapeEmbedding} \; \text{ on } \ShapeEmbedding(\Manifold).
		\end{equation}
		Since $\ShapeEmbedding$ is chosen from $\operatorname{Emb}_{\partial\Manifold}(\Manifold, \R^{n+1})$, it leaves the boundary fixed, finally giving '$(i) \implies (ii)$'.	

		Lastly, we easily see that '$(iii) \implies (i)$'.
		Since the complete vanishing of the pre-shape derivative \cref{Eq_Stationarity} in particular implies its vanishing for directions $V$ tangential to $\ShapeEmbedding(\Manifold)$, which is \cref{Eq_FiberCriticalPoint}, and concludes the proof.
	\end{proof}
%
	
\cref{Theorem_GlobalSolutionParamTrackingStationaryPoints} tells us that there are no stationary points other than global solutions to the pre-shape parameterization tracking problem  \cref{Eq_FinalGeneralParamTrackingProblem}.
This strongly resembles the situation for convex optimization problems, where the only candidates for local optimality are indeed global solutions.

Notice that \cref{Theorem_GlobalSolutionParamTrackingStationaryPoints} gives existence of stationary points $\ShapeEmbedding$ for each shape via existence result \cref{Theorem_ExistenceGeneralProblem}.
Since stationary points are global solutions, we can simply use the existence result \cref{Theorem_ExistenceGeneralProblem} for this.


Additionally, \cref{Theorem_GlobalSolutionParamTrackingStationaryPoints} guarantees that optimization with the tangential component of pre-shape derivative \cref{Eq_ShapeDerivTang} is sufficient to reach a globally optimal solution for (\ref{Eq_FinalGeneralParamTrackingProblem}).
This permits design of regularizations for shape optimization algorithms using pre-shape parameterization tracking with the property to leave the shape at hand invariant, while at the same time finding an optimal parameterization of the respective shape.

\section{Numerical Tests of Parameterization Tracking involving Pre-Shape Derivatives}\label{Section_NumericsGeneral}
We have now finished our introduction of pre-shape calculus and its application to parameterization tracking problems.
In order to test our theoretical results, we present three implementations of pre-shape gradient descent methods for the parameterization tracking problem.
For this we use the open-source finite-element software FEniCS (cf. \cite{LoggMardalEtAl2012a, AlnaesBlechta2015a}).
Construction of meshes is done via the free meshing software Gmsh (cf. \cite{geuzaine2007gmsh}).
We use a single core of an Intel(R) Core(Tm) i3-8100 CPU at 3.60 GHz featuring 16 GB RAM.
The single core runs at 800 MHz while the code is executed on a virtual machine.

In the following, we show three implementations solving the parameterization tracking problem \cref{Eq_FinalGeneralParamTrackingProblem} by using the tangential component of the pre-shape derivative seen in decomposition \cref{Eq_PrShpDerivativeTrackingDecompo}.
The solution process also features a simple backtracking line search, which scales the initial gradient of the current iteration $U_i$ according to a given factor $c$ and rescales it by $0.5$ if no sufficient decrease in $\TargetPrShp$ is apparent. 
In order to apply a descent algorithm, we are in need of pre-shape gradients. 
Because gradients are defined with respect to a sufficient bilinear form, we have to choose a form which fits our application.
Since we are in infinite dimensions, there is a multitude of non-equivalent choices to represent derivatives as gradients.
These can differ in resulting regularity of the gradients, and also in computational expense.
As a bilinear form, we choose the weak formulation of the linear elasticity as proposed in \cite{schulz2015Steklov}, which gives us $H^1$-regularity of pre-shape gradients.
By only using the shear component of the linear elasticity featuring the second Lam\'e parameter $\mu_{\text{elas}}$, and adding a zero order term, the gradient $U$ is calculated by solving its representing system\begin{align}\label{Eq_LinElasMetric}
\begin{split}
\alpha_{\text{LE}}\int_{\HoldAll}\mu_{\text{elas}}\cdot\epsilon(U):\epsilon(V)\;\diff x  + \alpha_{L^2}\cdot(U,V)_{L^2(\HoldAll)}  &= \PrShpDeriv\TargetPrShp^\tau(\ShapeEmbedding)[V] \qquad \forall V\in H^1_0(\HoldAll, \mathbb{R}^{n+1})\\
\epsilon(U) &= \frac{1}{2}(\nabla U^T + \nabla U)\\
\epsilon(V) &= \frac{1}{2}(\nabla V^T + \nabla V)\\
U &= 0 \qquad \text{ on } \partial\HoldAll.
\end{split}	
\end{align}
Here, we choose the weights $\alpha_{\text{LE}}, \alpha_{L^2} >0$. 
For $\mu_{\text{max}}, \mu_{\text{min}} > 0$, the second Lam\'e parameter $\mu_{\text{elas}}$ is chosen as the solution of the Poisson problem
\begin{equation}
\begin{split}
-\Delta \mu_{\text{elas}} &= 0 \qquad \;\;\; \text{in } \HoldAll \\
\mu_{\text{elas}} &= \mu_{\text{max}} \quad\, \text{on } \ShapeEmbedding(\Manifold) \\
\mu_{\text{elas}} &= \mu_{\text{min}} \quad\; \text{on } \partial\HoldAll.
\end{split}
\end{equation}
Solving \cref{Eq_LinElasMetric} on the entire hold-all domain $\HoldAll$ gives us a volume representation $U$ of the pre-shape derivative $\PrShpDeriv\TargetPrShp^\tau$.
The pre-shape gradient system \cref{Eq_LinElasMetric} is assembled in FEniCS and solved with a sparse LU method from PETSc used as a linear algebra backend.

The first example shows an application of the parameterization tracking problem to improve the quality of a given hold-all domain $\HoldAll = [0,1]^2 \subset \R^2$.
This is realized by using an unstructured $2$-dimensional volume mesh created via Gmsh featuring 4262 triangular cells and 2212 nodes.
Then we distort the mesh quality of this unstructured mesh by applying 
\begin{equation}
\ShapeEmbedding_0\left(
\begin{matrix}
x \\
y
\end{matrix}
\right)
=
\left(
\begin{matrix}
0.025\cdot\sin(25.5\cdot x) \\
0
\end{matrix}
\right)
\end{equation}
as a deformation to the interior of $\HoldAll$.
The deformed initial mesh $\ShapeEmbedding_0(\HoldAll)$ is depicted in \cref{Fig_GradientDesc2D}.
Notice that in this scenario the initial model $\Manifold$ is given by the hold-all domain $\HoldAll = [0,1]^2$ with non-trivial boundary $\partial\HoldAll$.
Therefore we are in the situation where the boundary $\partial\HoldAll$ is left invariant (cf. \cref{Eq_PreShapeSpaceNonTrivBoundary}).
Also, there is no normal component of the pre-shape derivative in this case, as the codimension of $\HoldAll\subset \R^2$ is zero.

To formulate the parameterization tracking problem \cref{Eq_FinalGeneralParamTrackingProblem}, we need to specify an initial point distribution $g^\Manifold$ and target $f_\ShapeEmbedding$.
Here, the target chosen is given by the constant 
\begin{equation}\label{Eq_Numerics_UniTarget}
f_\ShapeEmbedding \equiv \frac{1}{\int_{\HoldAll}1\; \diff x}.
\end{equation}
This ensures that a uniform cell volume distribution of the hold-all domain is targeted.
The initial point distribution $g^\Manifold$ is represented by using a continuous Galerkin Ansatz with linear elements. 
Degrees of freedom are situated at the mesh vertices and set to the average of inverses of surrounding cell volumes, i.e.
\begin{equation}
g^\Manifold(p_i) = \frac{1}{\vert \mathcal{C}_i\vert}\cdot\sum_{C\in \mathcal{C}_i}\frac{1}{\int_{C}1\; \diff x}.
\end{equation}
Here $p_i$ is a mesh vertex and $\mathcal{C}_i$ is the set of its neighboring cells $C$.
Finally, the resulting function is normed to satisfy the demanded normalization condition \cref{Assumption_Theorem1fNormed} of the parameterization tracking problem.
The initial point distribution estimated by this procedure is shown in \cref{Fig_GradientDesc2D} (a).

With both $g^\Manifold$ and $f_\ShapeEmbedding$ specified, the target $\TargetPrShp^\tau$ and its pre-shape derivative $\PrShpDeriv\TargetPrShp^\tau$ can be assembled.
For the gradient representation we use weights $\alpha_{\text{LE}} = 0.02$, $\alpha_{L^2} = 1$ and Lam\'e parameters $\mu_{\text{max}} = \mu_{\text{min}} = 1$, resulting in constant $\mu_{\text{elas}}=1$.
An initial scaling factor of $c= 0.01$ for the negative gradient during line search is applied.
The method successfully exits after 37.07s and 45 iterations.
Results of the pre-shape gradient descent using the tangential component of the pre-shape derivative and the described methodology are shown in \cref{Fig_GradientDesc2D} and \cref{Fig_NumericTargetGradNormsGraph}.

As our second example, we use the exact same parameters as in the first example.
Note that in particular, the starting mesh and therefore its initial volume distribution $g^\Manifold$ are the same as in the first example. 
We can see an illustration in \cref{Fig_GradientDesc2D} (a).
To show the general applicability of parameterization tracking, we replace the uniform target $f_\ShapeEmbedding$ from \cref{Eq_Numerics_UniTarget} by a more complicated non-uniform target
\begin{equation}
\hspace{-1cm}
f_\ShapeEmbedding=
\frac{\int_{[0,1]}\int_{[0,1]} g^\Manifold(x,y) \,\diff x\,\diff y}
{ \int_{[0,1]}\int_{[0,1]} 2 + \cos\Big(5\cdot 2\pi\cdot\big( (x - 0.35)^2 + 2\cdot(y-0.4)^2   \big)\Big) \,\diff x\,\diff y}
\cdot
\Bigg(2 + \cos\Big(5\cdot 2\pi\cdot\big( (x - 0.35)^2 + 2\cdot(y-0.4)^2   \big)\Big)\Bigg).
\end{equation}
The pre-shape gradient descent for this non-uniform target achieves convergence after 38.12s and 46 iterations.
We visualize an intermediate mesh, and the final mesh in \cref{Fig_GradientDesc2D} (c) and (d).
The target function values $\TargetPrShp^\tau(\ShapeEmbedding_i)$ and pre-shape gradient norms are shown in \cref{Fig_NumericTargetGradNormsGraph}.
Interestingly, notice that the intermediate mesh (c) looks like a superposition of the final and initial meshes (d) and (a).
Essentially, this is an illustration of snapshots from a discretized flow in the fiber of $\operatorname{Emb}(\HoldAll, \HoldAll)$ corresponding to the shape $\HoldAll$, which is abstractly visualized in \cref{Fig_EmbSpace}.
We see in \cref{Fig_GradientDesc2D} (d), that the prescribed non-uniform cell volume distribution is achieved, even though the initial mesh in \cref{Fig_GradientDesc2D} (d) has degenerate cells distributed on vertical lines.
\begin{figure}[h]
	\centering
	\begin{tabular}{cr}
		\begin{subfigure}{0.5\textwidth}
			\includegraphics[width=0.9\linewidth, height=5cm]{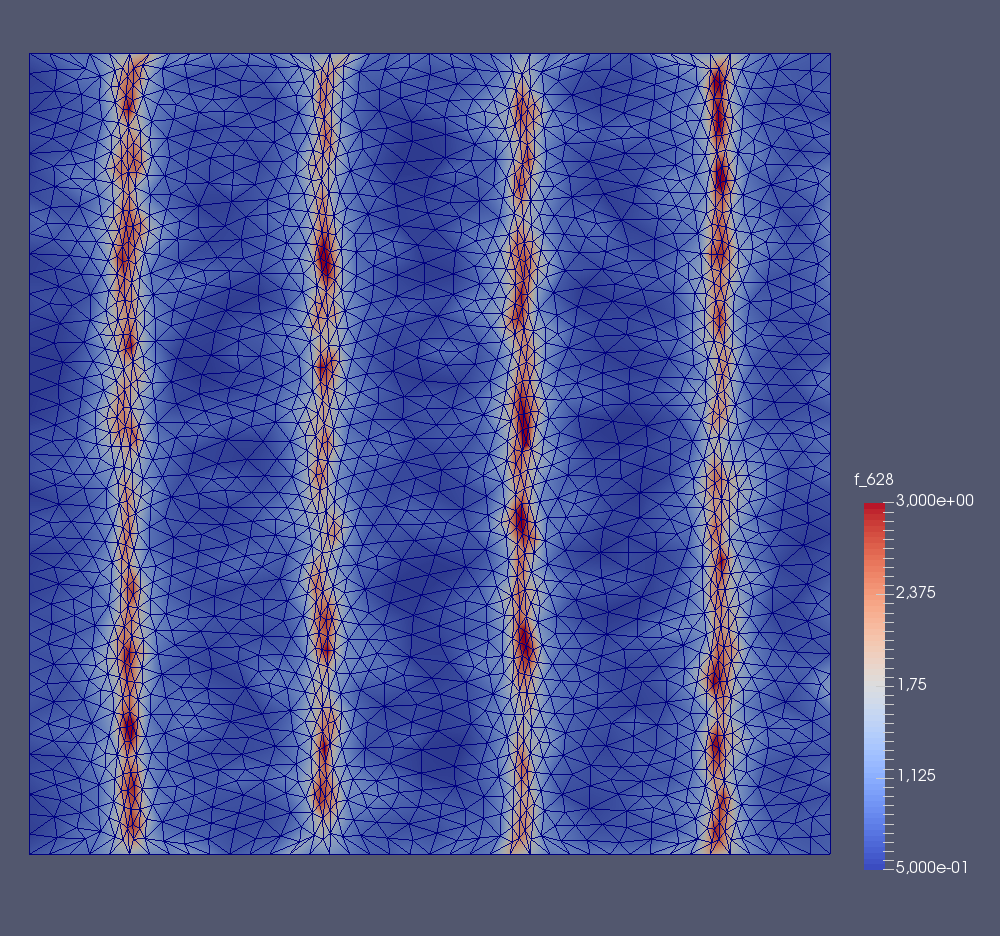}
			\subcaption{Initial mesh $\ShapeEmbedding_0(\Manifold)$}
		\end{subfigure} 
		&\begin{subfigure}{0.5\textwidth}
			\includegraphics[width=0.9\linewidth, height=5cm]{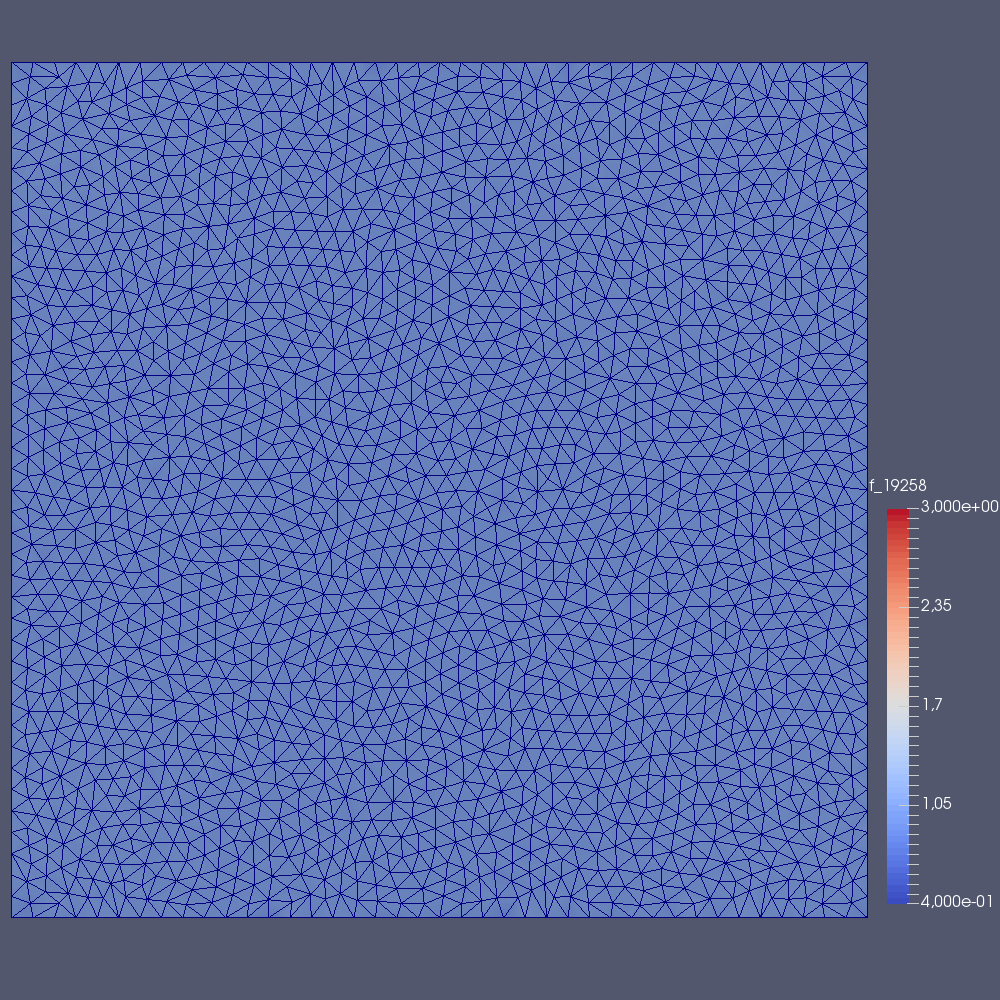}
			\subcaption{Final mesh $\ShapeEmbedding_{45}(\Manifold)$ for uniform target}
		\end{subfigure} 
		\\
		\begin{subfigure}{0.5\textwidth}
			\includegraphics[width=0.9\linewidth, height=5cm]{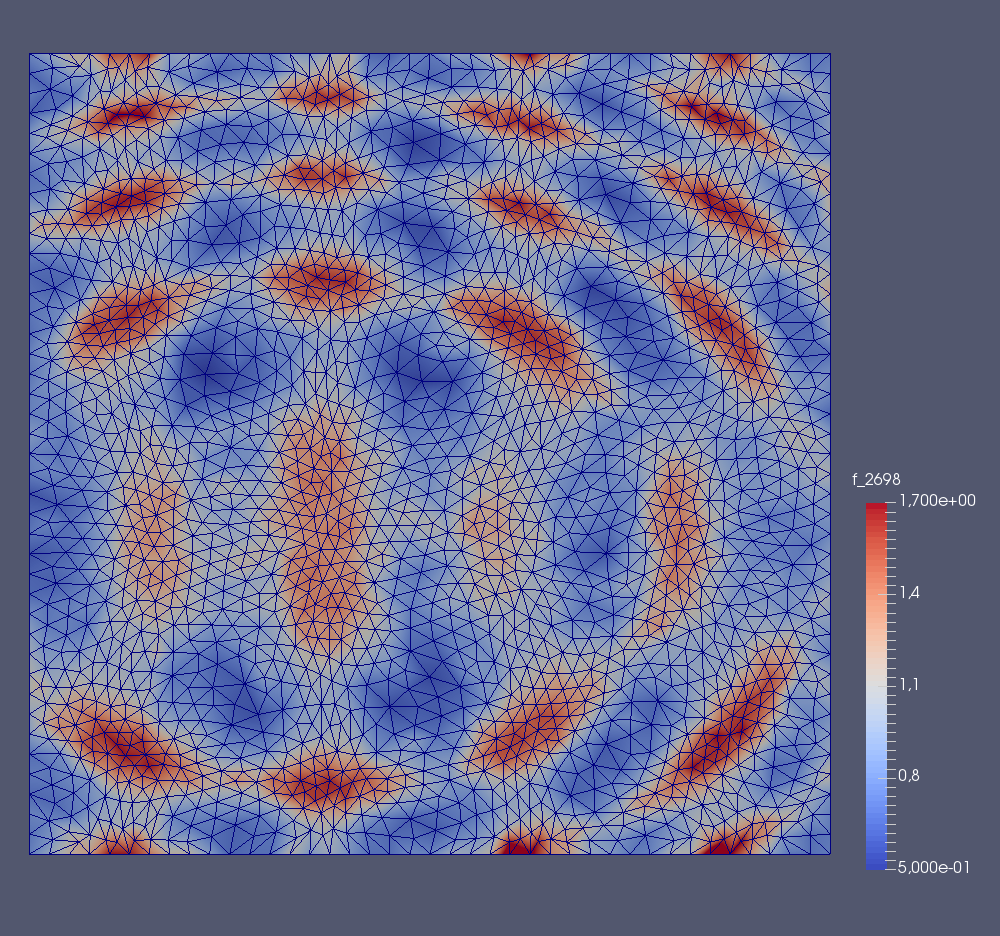}
			\subcaption{Intermediate mesh  $\ShapeEmbedding_{6}(\Manifold)$ for non-uniform target}
		\end{subfigure} 
		&\begin{subfigure}{0.5\textwidth}
			\includegraphics[width=0.9\linewidth, height=5cm]{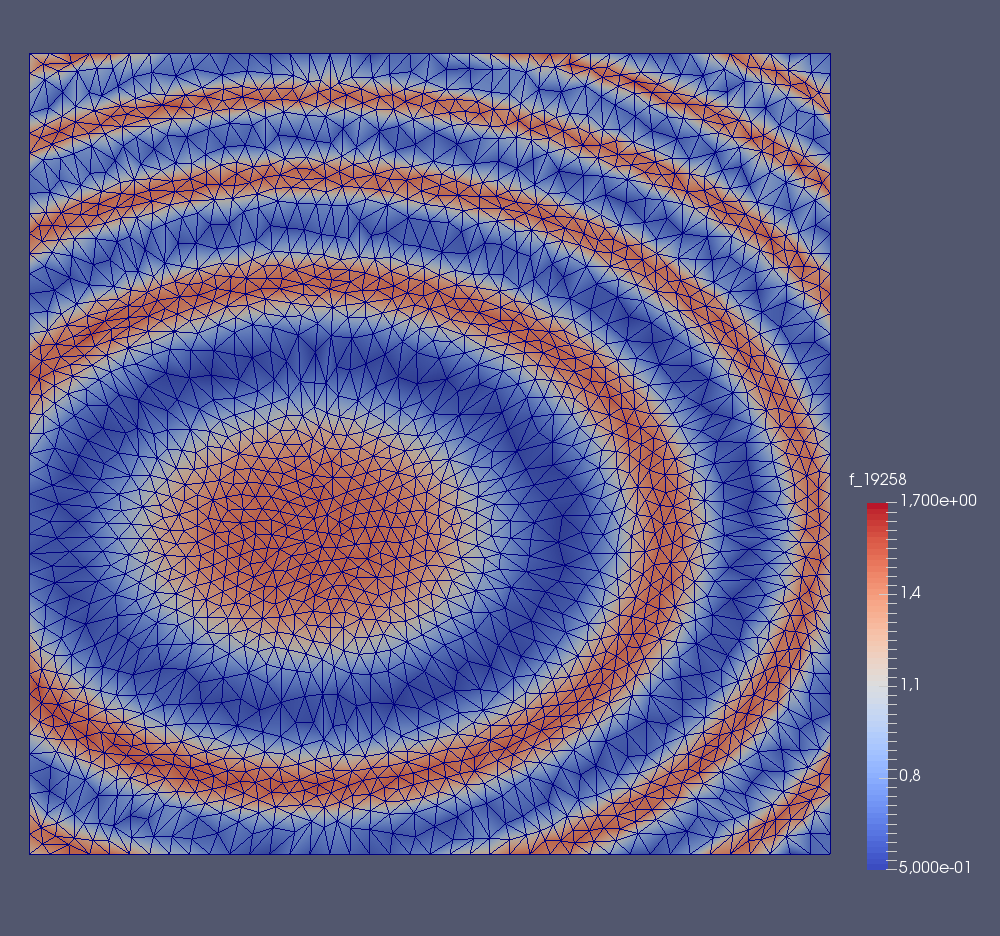}
			\subcaption{Final mesh  $\ShapeEmbedding_{46}(\Manifold)$ for non-uniform target}
		\end{subfigure} 
	\end{tabular}
	\caption{\label{Fig_GradientDesc2D}
		(a) Initial point distribution $g^\Manifold$ depicted by color on the distorted initial mesh $\ShapeEmbedding_0(\Manifold)$. \\
		(b) Final mesh $\ShapeEmbedding_{45}(\Manifold)$ for the uniform target after 45 pre-shape gradient descent iterations with associated point distribution $g^\Manifold\circ\ShapeEmbedding_{45}^{-1}\cdot \operatorname{det}D\ShapeEmbedding_{45}^{-1}$ shown in color. \\
		(c) Intermediate mesh $\ShapeEmbedding_{6}(\Manifold)$ for the non-uniform target after 6 pre-shape gradient descent iterations with associated point distribution $g^\Manifold\circ\ShapeEmbedding_{6}^{-1}\cdot \operatorname{det}D\ShapeEmbedding_{6}^{-1}$ shown in color. \\
		(d) Final mesh $\ShapeEmbedding_{46}(\Manifold)$ for the non-uniform target after 46 pre-shape gradient descent iterations with associated point distribution $g^\Manifold\circ\ShapeEmbedding_{46}^{-1}\cdot \operatorname{det}D\ShapeEmbedding_{46}^{-1}$ shown in color.}	
\end{figure}
Our third example applies the parameterization tracking problem to a sphere centered in the hold-all domain $\HoldAll = [0,1]^3\subset\R^{3}$.
It acts as the modeling manifold $\Manifold$ and its initial parameterization $\ShapeEmbedding_0$ is given by the identity embedding it into the hold-all domain.
The initial shape is a structured triangular surface mesh approximating a sphere centered in $(0.5, 0.5, 0.5)$ with radius $0.3$ using Gmsh. 
It consists of 6240 triangular cells and 3122 vertices on the surface.
The sphere is embedded in a hold-all domain consisting of 21838 tetraedic cells and 4059 nodes.

For the third example we target a non-uniform surface cell volume distribution given by
\begin{equation}
f_\ShapeEmbedding\left(
\begin{matrix}
x \\
y \\
z
\end{matrix}
\right)
=
1 + \frac{1}{2}\cdot\sin(10\cdot 2\pi\cdot x).
\end{equation}
The target function is of the form \cref{Eq_ExternalForceNormed}, which permits use of the material derivative formula \cref{Eq_MaterialDerivativefExtForce} for assembling the pre-shape derivative $\PrShpDeriv\TargetPrShp^\tau$.
At the same time, it satisfies normalization condition \cref{Assumption_Theorem1fNormed}.
Also, we set the initial vertex distribution to a constant
\begin{equation}
g^\Manifold \equiv \frac{1}{\int_{\Manifold} 1 \; \diff s}.
\end{equation}
In order to calculate covariant derivatives and associated Jacobian determinants, we apply a Gram-Schmidt algorithm to construct local tangential orthonormal bases.
Here, we choose weights $\alpha_{\text{LE}} = 0.02$, $\alpha_{L^2} = 1$ and Lam\'e parameters $\mu_{\text{max}} = 30$, $\mu_{\text{min}} = 5$ for gradient representation.
The line search employs an initial scaling factor $c= 0.001$ for the negative gradient.
For this scenario, the gradient representation of the pre-shape derivative, and the resulting surface mesh with its associated vertex distribution are depicted in \cref{Fig_GradientDesc3D}.
The method successfully exits after 1256.78s and 48 iterations.
Target function values $\TargetPrShp^\tau(\ShapeEmbedding_i)$ and pre-shape gradient norms are shown in \cref{Fig_NumericTargetGradNormsGraph}.
In light of \cref{Theorem_GlobalSolutionParamTrackingStationaryPoints}, we see that the gradient norm and target values converge simultaneously by using tangential components of $\PrShpDeriv\TargetPrShp^\tau$ only.
Also, the shape of the sphere is left invariant, which would not be the case if normal components or the full pre-shape derivative (cf. \cref{Fig_SurfGradDecompo}) were used.
\begin{figure}[h]
	\centering
	\begin{tabular}{cr}
		\begin{subfigure}{0.5\textwidth}
			\includegraphics[width=0.9\linewidth, height=5cm]{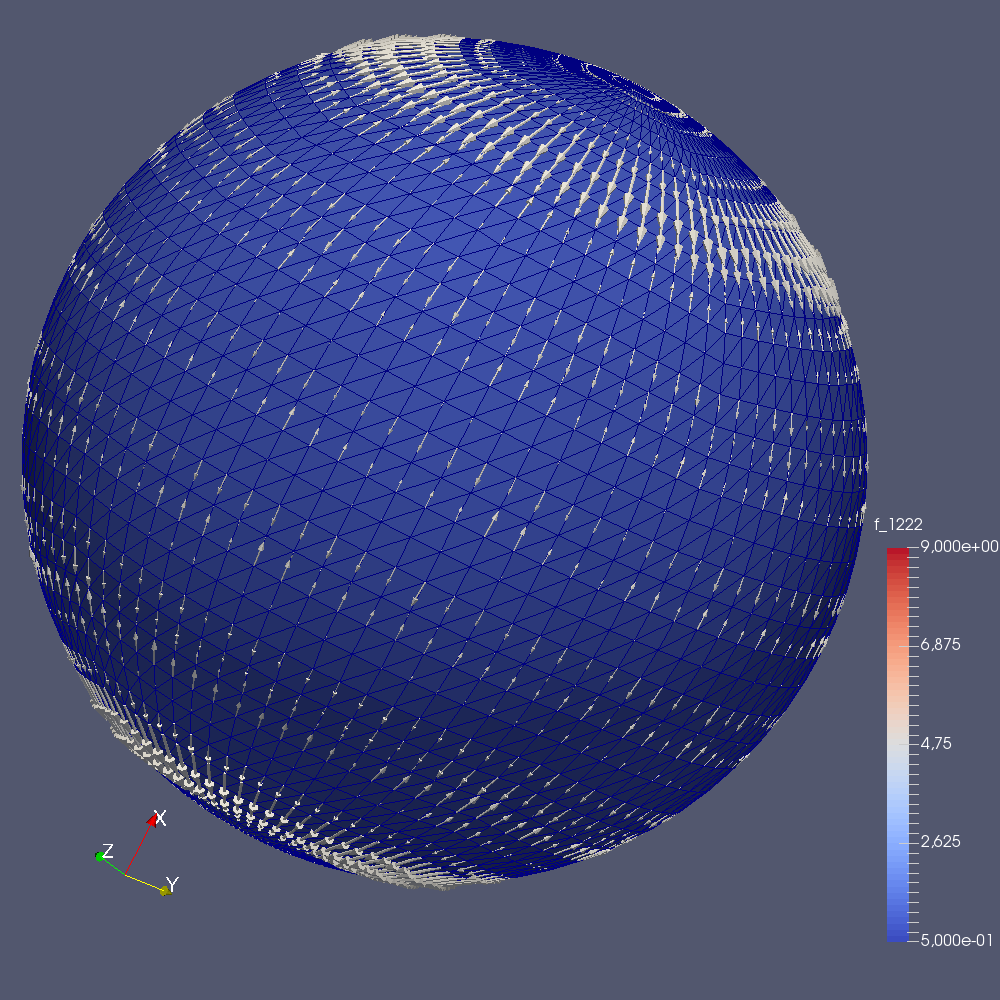}
			\subcaption{$g^\Manifold$ and $-U$ on $\ShapeEmbedding_0(\Manifold)$}
		\end{subfigure} 
		&\begin{subfigure}{0.5\textwidth}
			\includegraphics[width=0.9\linewidth, height=5cm]{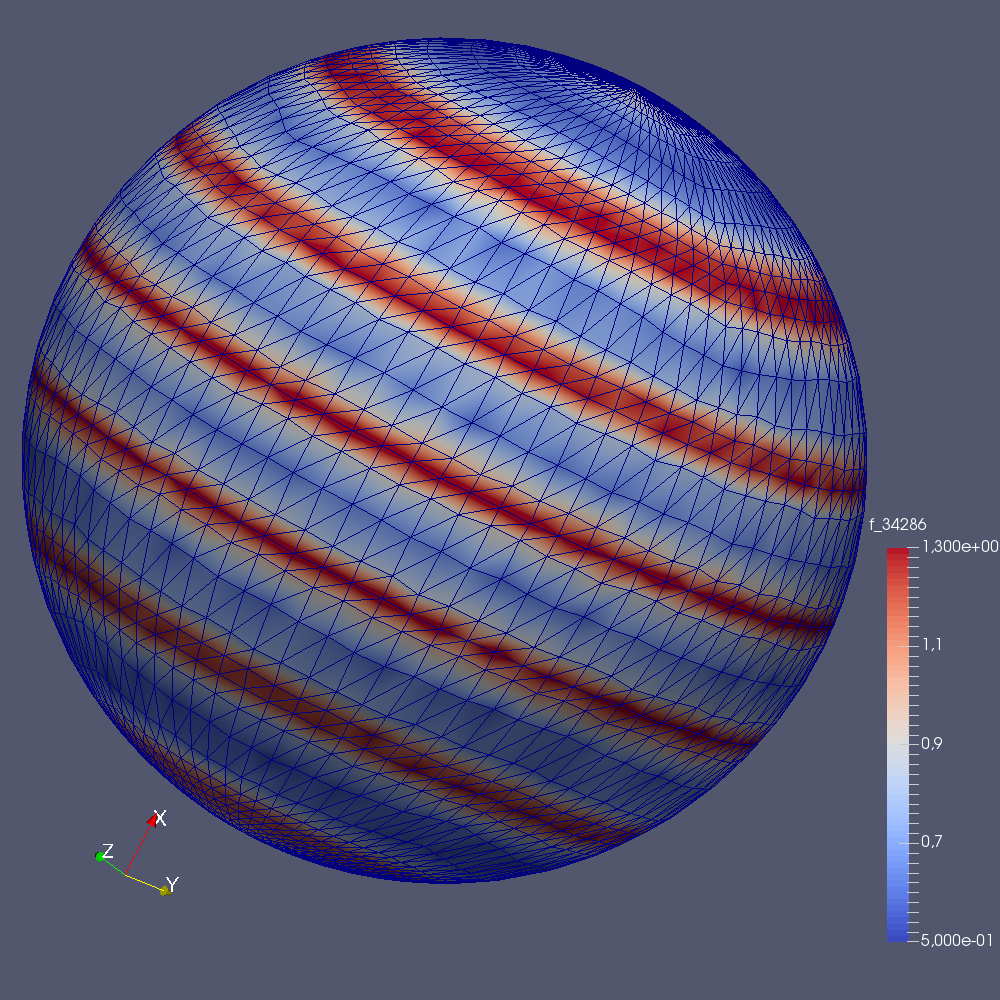}
			\subcaption{$g^\Manifold\circ\ShapeEmbedding_{48}^{-1}\cdot \operatorname{det}D\ShapeEmbedding_{48}^{-1}$ on  $\ShapeEmbedding_{48}(\Manifold)$}
		\end{subfigure} 
	\end{tabular}
	\caption{\label{Fig_GradientDesc3D}(a) Constant initial point distribution $g^\Manifold$ and negative pre-shape derivative $-\PrShpDeriv\TargetPrShp^\tau(\ShapeEmbedding_0)$ represented via \cref{Eq_LinElasMetric} on the initial surface mesh $\ShapeEmbedding_0(\Manifold)$ scaled by $0.03$. \\
		(b) Resulting surface mesh $\ShapeEmbedding_{48}(\Manifold)$ after 48 pre-shape gradient descent iterations with associated point distribution $g^\Manifold\circ\ShapeEmbedding_{48}^{-1}\cdot \operatorname{det}D\ShapeEmbedding_{48}^{-1}$ shown in color.}	
\end{figure}

\begin{figure}[h]
	\centering
	\begin{tabular}{cr}
		\begin{subfigure}{0.5\textwidth}
			\includegraphics[width=0.9\linewidth, height=5cm]{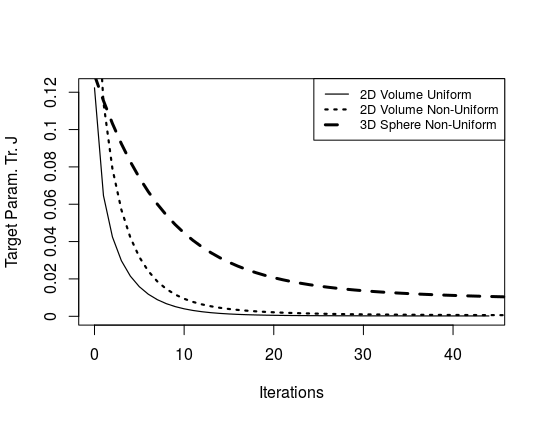}
			\subcaption{$\TargetPrShp^\tau(\ShapeEmbedding_i)$}
		\end{subfigure} 
		&\begin{subfigure}{0.5\textwidth}
			\includegraphics[width=0.9\linewidth, height=5cm]{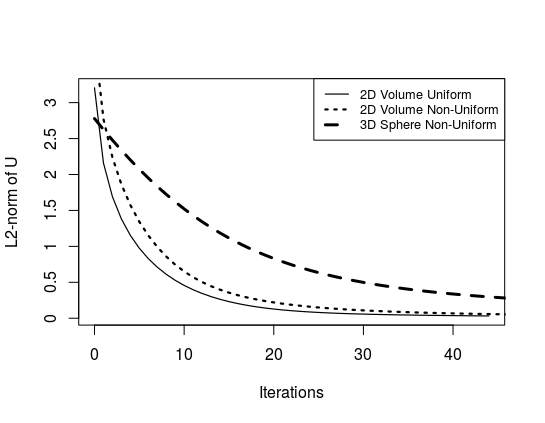}
			\subcaption{$\Vert U_i \Vert_{L^2(\HoldAll, \R^{n+1})}$}
		\end{subfigure} 
	\end{tabular}
	\caption{\label{Fig_NumericTargetGradNormsGraph}(a) Values for the pre-shape parameterization tracking target $\TargetPrShp^\tau(\ShapeEmbedding_i)$ for iterates  $\ShapeEmbedding_i$ of the tangential pre-shape derivative component based steepest descent method. Target for the 3D sphere case is scaled by 3. \\
		(b) $L^2$-norms $\Vert U_i \Vert_{L^2(\HoldAll, \R^{n+1})}$ of the gradient representations $U_i$ of pre-shape derivatives for each iterate $\ShapeEmbedding_i$.
		Gradient norms for the 3D sphere case are scaled by 25.}	
\end{figure}

\section{Conclusion and Outlook}
In this work we introduced a unified framework to formulate shape optimization and mesh quality optimization problems.
A calculus for pre-shape derivatives, which act in normal and tangential directions, and according structure theorems were derived.
In particular, rules and problem formulations from classical shape optimization carry over to the pre-shape setting.
These techniques were tested on a class of parameterization tracking problems.
Resulting numerical implementations of a gradient descent method based on decomposed pre-shape derivatives show promising results for optimization of volume- and surface mesh quality.

In forthcoming works we will derive efficient algorithms harnessing the opportunity to simultaneously solve shape optimization problems and improve mesh quality of shapes and ambient spaces.
For this, can will design various targets for parameterization tracking, giving a desired type of mesh for the user. 
Also, we will define pre-shape Hessians to harness second order information.
(Quasi-)Newton methods in the context of pre-shape optimization, as well as an optimal choice of pre-shape gradient representations, will enhance the performance of these algorithms.

\section*{Acknowledgements}
The authors would like to thank Leonhard Frerick (Trier University) and Jochen Wengenroth (Trier University) for a helpful and interesting discussion about differentiability in infinite dimensions.
This work has been supported by the BMBF (Bundesministerium f\"{u}r Bildung und Forschung) within the collaborative project GIVEN (FKZ: 05M18UTA).
Further, the authors acknowledge the support of the DFG research training group 2126 on algorithmic optimization.


\bibliographystyle{plain}
\bibliography{citations.bib}

\end{document}